\documentclass[reqno,11pt]{amsart}

\usepackage{latexsym}
\usepackage{float}
\usepackage{amssymb}
\usepackage{amsmath}
\usepackage{amsthm}
\usepackage{mathrsfs}
\usepackage{euscript}
\usepackage[cmtip,all]{xy}

\usepackage{bbold}
\usepackage{tabu}

\makeatletter
\g@addto@macro \normalsize {%
 \setlength\abovedisplayskip{7pt}%
 \setlength\belowdisplayskip{6pt}%
}
\makeatother

\newtheorem{thm}[equation]{Theorem}
\newtheorem{lem}[equation]{Lemma}
\newtheorem{cor}[equation]{Corollary}
\newtheorem{prop}[equation]{Proposition}

\theoremstyle{remark}
\newtheorem{rem}[equation]{Remark}
\newtheorem{fact}[equation]{Fact}

\newtheorem{defn}[equation]{Definition}

\numberwithin{equation}{section}

\newcommand{\gb}{\beta}
\newcommand{\ga}{\alpha}

\newcommand{\gl}{\nu}

\newcommand{\gD}{\Delta}

\newcommand{\gt}{\theta}

\newcommand{\gS}{\Sigma}

\newcommand{\eps}{\varepsilon}

\newcommand{\fa}{{\mathfrak a}}             
\newcommand{\fb}{{\mathfrak b}}
\newcommand{\fg}{{\mathfrak g}}
\newcommand{\fh}{{\mathfrak h}}
\newcommand{\fk}{{\mathfrak k}}

\newcommand{\fm}{{\mathfrak m}}
\newcommand{\fn}{{\mathfrak n}}

\newcommand{\fp}{{\mathfrak p}}

\newcommand{\ft}{{\mathfrak t}}

\newcommand{\fs}{{\mathfrak s}}
\newcommand{\fy}{{\mathfrak y}}

\newcommand{\f}{\mathfrak}

\newcommand{\cga}{\alpha^\ssv}

\newcommand{\R}{\mathbb{R}}          
\newcommand{\C}{\mathbb{C}}          
           
\newcommand{\Z}{\mathbb{Z}}

\newcommand{\Ad}{\mathrm{Ad}}
\newcommand{\Cal}{\mathcal}

\newcommand{\Hom}{\operatorname{Hom}}

\renewcommand{\Im}{\mathrm{Im}}

\newcommand{\IP}[2]{\langle#1 , #2\rangle}     

\newcommand{\Sol}{\mathrm{Sol}}

\newcommand{\Pol}{\mathrm{Pol}}
\newcommand{\To}{\longrightarrow}

\newcommand{\Diff}{\mathrm{Diff}}
\newcommand{\Irr}{\mathrm{Irr}}
\newcommand{\Rep}{\mathrm{Rep}}
\newcommand{\triv}{\mathrm{triv}}
\newcommand{\fin}{\mathrm{fin}}
\newcommand{\sym}{\mathrm{sym}}

\newcommand{\D}{\Cal{D}}

\newcommand{\wG}{\widetilde{G}}
\newcommand{\wK}{\widetilde{K}}
\newcommand{\wM}{\widetilde{M}}
\newcommand{\wB}{\widetilde{B}}
\newcommand{\wm}{\widetilde{m}}
\newcommand{\wchi}{\widetilde{\chi}}
\newcommand{\wrho}{\widetilde{\rho}}
\newcommand{\wSL}{\widetilde{SL}}

\newcommand{\dpin}{d\pi_n}
\newcommand{\xvarphi}{\varphi_\xi}

\newcommand{\pp}{\textnormal{\mbox{\smaller($+$,$+$)}}}
\newcommand{\pmi}{\textnormal{\mbox{\smaller($+$,$-$)}}}
\newcommand{\mip}{\textnormal{\mbox{\smaller($-$,$+$)}}}
\newcommand{\mm}{\textnormal{\mbox{\smaller($-$,$-$)}}}
\newcommand{\pmpm}{\textnormal{\mbox{\smaller($\pm$,$\pm$)}}}

\newcommand{\ssv}{{\scriptscriptstyle \vee}}

\newcommand{\XcY}{X\circ Y}

\providecommand*{\donothing}[1]{}

\begin{document}

\baselineskip=16pt
\tabulinesep=1.2mm


\title[]{On the space of $K$-finite solutions to 
intertwining differential operators}

\author{Toshihisa Kubo}
\author{Bent {\O}rsted}

\address{Faculty of Economics, 
Ryukoku University,
67 Tsukamoto-cho, Fukakusa, Fushimi-ku, Kyoto 612-8577, Japan}
\email{toskubo@econ.ryukoku.ac.jp}

\address{Department of Mathematics, 
Aarhus University,
Ny Munkegade 118 DK-8000 Aarhus C Denmark}
\email{orsted@imf.au.dk}



\maketitle


\begin{abstract} 
In this paper we give Peter--Weyl type formulas for
the space of $K$-finite solutions to intertwining differential operators
between degenerate principal series representations.
Our results generalize a result of Kable for 
conformally invariant systems.
The main idea is based on the duality theorem between
intertwining differential operators and homomorphisms between
generalized Verma modules.
As an application we uniformly realize 
on the solution spaces of intertwining differential operators
all small representations
of $\widetilde{SL}(3,\R)$ attached to the minimal nilpotent orbit.
\end{abstract}


\section{Introduction}\label{sec:intro}


Let $G$ be a real reductive Lie group and
$P$ a parabolic subgroup of $G$. 
Given  
finite-dimensional representations $W$ and $E$ of $P$, 
write $\Cal{W}=G\times_{P}W$ and 
$\Cal{E} = G \times_{P} E$, the homogeneous vector bundles over $G/P$ with fibers
$W$ and $E$, respectively. 
The aim of this paper 
is to understand the representation realized on the solution space of
an intertwining differential operator
$\Cal{D}\colon C^\infty(G/P, \Cal{W}) \to C^\infty(G/P,\Cal{E})$ between 
the spaces of smooth sections for $\Cal{W}$ and $\Cal{E}$.

Realization of a representation on the space of solutions
to intertwining differential operators 
has been studied by a number of people such as
Kostant (\cite{Kostant90}),
Binegar--Zierau (\cite{BZ91, BZ94}),
{\O}rsted (\cite{Orsted00}),
Kobayashi--{\O}rsted (\cite{KO03a}, \cite{KO03b}, \cite{KO03c}),
Kable (\cite{Kable11}, \cite{Kable12A}, \cite{Kable12C}, \cite{Kable12B}),
Wang (\cite{Wang05}),
Sepanski and his collaborators (\cite{FS13}, \cite{HSS12}, \cite{HSS15}, \cite{SF15}),
among others.
For instance, in \cite{KO03a, KO03b, KO03c}, Kobayashi--{\O}rsted realized
the minimal representation of $O(p,q)$ in the solution space to the Yamabe operator
and studied it in great depth from the 
various perspectives of
conformal geometry, branching law, and harmonic analysis,
whereas Kable in \cite{Kable12A} used the Peter--Weyl theorem
to realize the minimal representation of $G_2$ 
on the common solution space 
of a system of differential operators constructed in \cite{BKZ08}.
In this paper we take the approach of Kable to
understand the $K$-type formula of the representations realized on 
the space of $K$-finite solutions to intertwining differential operators.
For convenience we refer to a general $K$-type decomposition formula 
such as \eqref{eqn:K-intro} below as a \emph{Peter--Weyl type formula}
(or \emph{PW type formula} for short).


In \cite{Kable11}, (for linear group $G$),
Kable gave a Peter--Weyl type formula for the space 
of $K$-finite solutions to a system of differential operators 
that are equivariant under an action of 
the Lie algebra of $G$.
Such a system of operators is called a \emph{conformally invariant system}
(\cite{BKZ08, BKZ09}).
From the viewpoint of intertwining operators
a conformally invariant system is
an intertwining differential operator from a line bundle to a vector bundle.
In this paper
we give a Peter--Weyl type formula 
for the space of $K$-finite solutions to
intertwining differential operators from general vector bundles
(Theorem \ref{prop:SVW});
the case from a line bundle to a vector bundle
is also further investigated
(Proposition \ref{prop:PWSol} and Theorems \ref{thm:intro1} and \ref{thm:Ksol}).
In addition we also consider
the common solution space of a system of intertwining differential operators 
(Section \ref{subsec:sol}).

Our main tools are the Peter--Weyl theorem
for the space $C^\infty(G/P, \Cal{W})_K$ of $K$-finite sections (\eqref{eqn:PW})
and the duality theorem between intertwining differential operators and homomorphisms 
between generalized Verma modules (Theorem \ref{thm:duality}).
There are three main cases.

\begin{enumerate}
\item VV case: a vector bundle to a vector bundle (Theorem \ref{prop:SVW})
\item LV case: a line bundle to a vector bundle  \hspace{0.4cm} 
(Theorem \ref{thm:Ksol})
\item LL case: a line bundle to a line bundle \hspace{0.85cm} 
(Theorem \ref{thm:intro1})
\end{enumerate}
We next briefly describe the Peter--Weyl type formula for the LL case.


Let $G$ and $P$ be as above and fix a maximal compact subgroup $K$ of $G$.
Write $P=MAN$ for a Langlands decomposition of $P$.
Let $\fg_0$ be the Lie algebra of $G$ and we denote by
$\fg$ and $\Cal{U}(\fg)$ the complexification of $\fg_0$ and the 
universal enveloping algebra of $\fg$, respectively.
A similar convention is employed also for the subgroups $K$ and $P$.
Given characters $\chi_{\triv}, \chi$ of $M$ with $\chi_{\triv}$ the trivial character,
and also $\lambda, \nu$ of $A$,
let $\C_{\triv,\lambda}$ (resp.\ $\C_{\chi,\nu}$)
denote the one-dimensional representation
$(\chi_{\triv}\otimes (\lambda+\rho), \C)$ 
(resp.\ $(\chi\otimes (\nu+\rho), \C)$)
of $P=MAN$ with trivial $N$ action,
where $\rho$ is half the sum of the positive roots.
We write $\Cal{L}_{\triv, \lambda}$ and $\Cal{L}_{\chi, \nu}$
for the line bundles over $G/P$ with fibers 
$\C_{\triv, \lambda}$ and $\C_{\chi, \nu}$, respectively.
We realize 
the degenerate principal series representation
$I_P(\chi_\triv, \lambda)$ 
on the space of smooth sections
$C^\infty(G/P, \Cal{L}_{\triv, \lambda})$ 
for $\Cal{L}_{\triv, \lambda}$.
The representation
$I_P(\chi, \nu)$ is defined similarly.
We write $\Diff_G(I_P(\chi_\triv, \lambda), I_P(\chi, \nu))$ for
the space of intertwining differential operators 
$\Cal{D}\colon C^\infty(G/P, \Cal{L}_{\triv, \lambda}) 
\to C^\infty(G/P,\Cal{L}_{\chi, \nu})$.
It follows from the duality theorem that
any differential operator
$\Cal{D} \in \Diff_G(I_P(\chi_\triv, \lambda), I_P(\chi, \nu))$ 
is of the form $\Cal{D} = R(u)$
for some $u \in \Cal{U}(\fg)$, where $R$ denotes the infinitesimal 
right translation of $\Cal{U}(\fg)$. 
To emphasize the element $u$, we write $\Cal{D}_u$ 
for the differential operator such that $\Cal{D}_u = R(u)$.

Let $\Irr(K)$ and $\Irr(M/M_0)$
be the sets of equivalence classes of irreducible representations 
of $K$ and the component group $M/M_0$ of $M$,
respectively.
It follows from Lemma \ref{lem:tensor2} in Section \ref{subsec:TOP}
that, for $\Cal{D}_u \in \Diff_G(I_P(\chi_\triv, \lambda), I_P(\chi, \nu))$
and $\xi \in \Irr(M/M_0)$, we have
$\Cal{D}_u\otimes \mathrm{id}_{\xi} \in
\Diff_G(I_P(\xi, \lambda), I_P(\chi\otimes\xi, \nu))$.
We remark that the representation $\xi \in \Irr(M/M_0)$ 
needs not be a character, 
as $G$ is not necessarily linear 
(see \eqref{eqn:M-irred} below for the case that $M=M/M_0$).

For $V_\delta:=(\delta, V) \in \Irr(K)$ and $u \in \Cal{U}(\fg)$,
we define a subspace $\Sol_{(u)}(\delta)$ of $V_\delta$ by
\begin{equation}\label{eqn:SolK-intro}
\Sol_{(u)}(\delta):= \{v \in V_\delta :
d\delta(\tau( u^\flat) )v = 0 \}.
\end{equation}
Here $d\delta$ denotes the differential of $\delta$, $\tau$ denotes the 
conjugation of $\fg$ with respect to $\fg_0$,
and $u^\flat$ is the element of $\Cal{U}(\fk)$ such that 
$u^\flat \otimes \mathbb{1}_{-\lambda-\rho} = u \otimes \mathbb{1}_{-\lambda-\rho}$ 
in $\Cal{U}(\fg) \otimes_{\Cal{U}(\fp)}\C_{-\lambda-\rho}$.
It will be shown that the space 
$\Sol_{(u)}(\delta)$ is a $K\cap M$-representation 
(see Lemma \ref{lem:SolM}).

Given $\Cal{D}_u \in \Diff_G(I_P(\chi_\triv, \lambda), I_P(\chi, \nu))$
and $\xi \in \Irr(M/M_0)$,
we write $\Cal{S}ol_{(u; \lambda)}(\xi)_K$
for the space of the $K$-finite solutions to $\Cal{D}_u \otimes \mathrm{id}_\xi$.
With the notation we obtain the following
as a specialization of Theorem  \ref{thm:Ksol}.
(For some details see Section \ref{subsec:recipe1}.)

\begin{thm}
[PW type formula for the LL case]
\label{thm:intro1}
Let $\Cal{D}_u \in \Diff_G(I_P(\chi_\triv, \lambda), I_P(\chi, \nu))$
and $\xi \in \Irr(M/M_0)$.
Then
the space $\Cal{S}ol_{(u; \lambda)}(\xi)_K$
of $K$-finite solutions to 
$\Cal{D}_u\otimes \mathrm{id}_{\xi}$
is decomposed as a $K$-representation as
\begin{equation}\label{eqn:K-intro}
\Cal{S}ol_{(u; \lambda)}(\xi)_K
\simeq \bigoplus_{\delta \in \Irr(K)}
V_\delta \otimes \mathrm{Hom}_{K\cap M}\left(
\Sol_{(u)}(\delta), \xi\right).
\end{equation}
\end{thm}


As an application of Theorem \ref{thm:intro1},
we take $G$ to be 
$\wSL(3,\R)$, the non-linear double cover of $SL(3,\R)$, 
and $P$ to be a minimal parabolic $\wB$ of $\wSL(3,\R)$. 
Write $\wB =\wM AN$ for a Langlands decomposition of $\wB$.
Here $\wM$ is isomorphic to the quaternion group $Q_8$, 
a non-commutative group of order 8.
In particular, $\wM$ is a discrete subgroup of $\wSL(3,\R)$ 
so that $\wM = \wM/\wM_0$.
As $\wM \simeq Q_8$,
the set $\Irr(\wM)$ is given by
\begin{equation}\label{eqn:M-irred}
\Irr(\wM)= \{ \pp, \pmi, \mip, \mm, \mathbb{H}\},
\end{equation}
where $\pmpm$ are characters and $\mathbb{H}$ is the unique 2-dimensional 
irreducible representation of $Q_8\simeq \wM$.
(For the notation $\pmpm$, see Section \ref{sec:SL3}.)

In this setting we consider two cases, namely, 
the case for infinitesimal character $\rho$
and that for infinitesimal character $\wrho:=(1/2)\rho$.
For each case we take $\lambda$ in \eqref{eqn:K-intro} 
to be $\lambda = -\rho$ and $\lambda=-\wrho$, respectively.
Via the duality theorem 
we obtain first-order operators $\Cal{D}_X, \Cal{D}_Y$, 
third-order operators $\Cal{D}_{Y^2X}, \Cal{D}_{X^2Y}$, 
and fourth-order operator $\Cal{D}_{XY^2X} (=\Cal{D}_{YX^2Y})$ for the $\rho$ case, 
and a second-order operator $\Cal{D}_{\XcY}$ is obtained for the $\wrho$ case.
(For the notation $X, Y$, and $\XcY$, see \eqref{eqn:XY} and \eqref{eqn:XcY}.)
We remark that all operators are also constructed 
via the BKZ-construction (\cite{BKZ08, Kubo11}).
Moreover, the second-order operator $\Cal{D}_{\XcY}$ is 
a specialization of Kable's Heisenberg ultrahyperbolic operator 
(\cite{Kable12C, Kable12B}), which is used in \cite{Kable12C} 
to establish a Heisenberg analogue 
of classic Maxwell's theorem on harmonic polynomials on Euclidean space.

For the sake of simplicity, we write
$\D_u^\sigma = \D_u \otimes \mathrm{id}_\sigma$
for $\sigma \in \Irr(\wM)$.
It is easily observed that
the solution space $\Cal{S}ol_{(X;-\rho)}(\sigma)$ of 
$\Cal{D}^\sigma_X$ 
(resp.~ $\Cal{S}ol_{(Y;-\rho)}(\sigma)$ of 
$\Cal{D}^\sigma_{Y}$)
is contained in that of $\Cal{D}^\sigma_{Y^2X}$ 
and  $\Cal{D}^\sigma_{XY^2X}$
(resp.\ $\Cal{D}^\sigma_{X^2Y}$ and  $\Cal{D}^\sigma_{YX^2Y}$).
Then, in this paper, we focus on the solution spaces
of $\Cal{D}^\sigma_X$ and $\Cal{D}^\sigma_Y$ for the $\rho$ case.
Further the common solution space 
$\Cal{S}ol_{(X, Y;-\rho)}(\sigma)$ 
of $\Cal{D}^\sigma_X$ and $\Cal{D}^\sigma_Y$
is also investigated.

On the solution spaces of the first-order operators
we realize a number of irreducible representations studied in
\cite{HTY}.
In order to state the results
let $\wK$ be a maximal compact subgroup of $\wSL(3,\R)$.
As $\wK \simeq SU(2) \simeq Spin(3)$,
the irreducible representations $\delta \in \Irr(\wK)$ of $\wK$ can be parametrized
as $\Irr(\wK) \simeq \{V_{(\frac{a}{2})} : a \in \Z_{\geq 0} \}$,
where $V_{(\frac{a}{2})}$ is the irreducible representation of $\wK$
with $\dim_\C V_{(\frac{a}{2})} = a+1$.
Then, for $u = X, Y, (X,Y)$,
the classification of 
$\sigma \in \Irr(\wM)$ such that $\Cal{S}ol_{(u;-\rho)}(\sigma) \neq \{0\}$ 
and the $\wK$-type formula of the space 
$\Cal{S}ol _{(u;-\rho)}(\sigma)_{\wK}$
of $\wK$-finite solutions to $\D^\sigma_{u}$ are obtained as follows.

\begin{thm}
\label{thm:Ktype-rho}
For $\sigma \in \Irr(\wM)$, the following hold.
\begin{enumerate}
\item
$\Cal{S}ol_{(X;-\rho)}(\sigma) 
\hspace{0.3cm}
\neq \{0\}  \iff \sigma = \pp, \pmi$.
\item
$\Cal{S}ol_{(Y;-\rho)}(\sigma) 
 \hspace{0.33cm}
\neq \{0\} \iff \sigma = \pp, \mip$.
\item
$\Cal{S}ol_{(X,Y;-\rho)}(\sigma) \neq \{0\} \iff \sigma = \pp$.
\end{enumerate}
Moreover, for $\sigma \in \Irr(\wM)$ 
such that $\Cal{S}ol _{(u;-\rho)}(\sigma)\neq \{0\}$,
the $\wK$-type formula of 
$\Cal{S}ol _{(u;-\rho)}(\sigma)_{\wK}$ is determined as follows.
\begin{enumerate}
\item[(a)] $u=X:$\hspace{0.75cm}
$\displaystyle{\Cal{S}ol _{(X;-\rho)}(\pp)_{\wK} 
\hspace{0.27cm}
\simeq \bigoplus_{a=0}^\infty V_{(2a)}}$ 
\; and \;
$\displaystyle{\Cal{S}ol _{(X;-\rho)}(\pmi)_{\wK} 
\simeq \bigoplus_{a=0}^\infty V_{(2a+1)}}$.
\vskip 0.05in

\item[(b)] $u=Y:$\hspace{0.8cm}
$\displaystyle{\Cal{S}ol _{(Y;-\rho)}(\pp)_{\wK} 
\hspace{0.3cm}
\simeq \bigoplus_{a=0}^\infty V_{(2a)}}$
\; and \;
$\displaystyle{\Cal{S}ol _{(Y;-\rho)}(\mip)_{\wK} 
\simeq \bigoplus_{a=0}^\infty V_{(2a+1)}}$.
\vskip 0.05in

\item[(c)] $u=(X,Y):$
$\displaystyle{\Cal{S}ol _{(X, Y;-\rho)}(\pp)_{\wK} \simeq V_{(0)}}$.
\end{enumerate}
\end{thm}

We prove Theorem \ref{thm:Ktype-rho} at the end of
Section \ref{sec:rho} (see Theorem \ref{thm:XandY}). 
We remark that although the $\wK$-type formulas for 
$\Cal{S}ol _{(X;-\rho)}(\pp)_{\wK}$ and 
$\Cal{S}ol _{(Y;-\rho)}(\pp)_{\wK}$,
and those for
$\Cal{S}ol _{(X;-\rho)}(\pmi)_{\wK}$ and
$\Cal{S}ol _{(Y;-\rho)}(\mip)_{\wK}$
are the same,
these spaces are different as $(\fg, \wK)$-modules.
Moreover, for $u=X,Y$, we set
\begin{equation*}
\Cal{S}ol _{(u/(X,Y);-\rho)}(\pp)_{\wK}
:=\Cal{S}ol _{(u;-\rho)}(\pp)_{\wK}/ \Cal{S}ol _{(X, Y;-\rho)}(\pp)_{\wK}.
\end{equation*}
Then the four representations
$\Cal{S}ol _{(X/(X,Y);-\rho)}(\pp)_{\wK}$,
$\Cal{S}ol _{(Y/(X,Y);-\rho)}(\pp)_{\wK}$,
$\Cal{S}ol _{(X;-\rho)}(\pmi)_{\wK}$, and
$\Cal{S}ol _{(Y;-\rho)}(\mip)_{\wK}$
are all irreducible $(\fg, \wK)$-modules.
(For the remarks see, for instance, \cite{HTY}.)

For the $\wrho$ case, 
we successfully realize all small representations of $\wSL(3,\R)$ 
attached to the minimal nilpotent orbit in the solution space 
$\Cal{S}ol_{(\XcY;-\wrho)}(\sigma)$ of
$\Cal{D}^\sigma_{\XcY}$,
one of which is so-called Torasso's representation.
Here is the main result for the second-order operator $\Cal{D}_{\XcY}$.

\begin{thm}
\label{thm:XcY}
For $\sigma \in \Irr(\wM)$, we have
\begin{equation*}
\Cal{S}ol_{(\XcY;-\wrho)}(\sigma) \neq \{0\} 
\iff \sigma = \pp,\, \mathbb{H},\, \mm.
\end{equation*}
Moreover, for $\sigma = \pp,\, \mathbb{H},\, \mm$, the $\wK$-type formula 
of $\Cal{S}ol _{(\XcY;-\wrho)}(\sigma)_{\wK}$ is obtained as follows.
\begin{enumerate}
\item[(a)] $\sigma = \pp:$ 
$\displaystyle{\Cal{S}ol _{(\XcY;-\wrho)}(\pp)_{\wK} \simeq \bigoplus_{a=0}^\infty V_{(2a)}}$.

\item[(b)] $\sigma= \mathbb{H}:$\hspace{0.6cm}
$\displaystyle{\Cal{S}ol _{(\XcY;-\wrho)}(\mathbb{H})_{\wK} \hspace{0.7cm}
\simeq \bigoplus_{a=0}^\infty V_{(2a+\frac{1}{2})}}$.
\item[(c)] $\sigma= \mm:$ 
$\displaystyle{\Cal{S}ol _{(\XcY;-\wrho)}(\mm)_{\wK} 
\simeq \bigoplus_{a=0}^\infty V_{(2a+1)}}$.
\end{enumerate}
\end{thm}

We would like to note that the set of $\wK$-types
$\bigoplus_{a=0}^\infty V_{(2a+\frac{3}{2})}$ is missing mysteriously.
A similar observation was also made in \cite{RS82} and 
\cite[Ex.~12.4]{Vogan91}.

The proof of Theorem \ref{thm:XcY} is given in
Section \ref{sec:wrho} (see Theorem \ref{thm:XcY1}).
It is remarked that 
one can also read off the results of (a) and (c) from \cite[Thm.~5.13]{Kable12C},
as these two cases are factored through $SL(3,\R)$. 
The spherical representation in (a) is also recently realized in 
\cite{HPP17} as the range of a residue operator of $SL(3,\R)$.

The representation obtained in the case of $\sigma = \mathbb{H}$ is Torasso's representation.
As Torasso's representation is the unique genuine irreducible representation of 
$\wSL(3,\R)$ attached to the minimal nilpotent orbit,
it has been widely studied from various different points of view.
See, for instance, \cite{Lucas08, Orsted00, RS82, Sijacki75, Tamori17, Torasso83, Tsai14},
as related works.
We provide another realization of the genuine representation,
which seems rather more elementary than any other realization in the literature.


In turn to the general theory observe that
in order to determine the $K$-type formula of solution spaces
via the isomorphism \eqref{eqn:K-intro}, 
one needs to solve the equation $d\delta(\tau( u^\flat) )v = 0$ 
in \eqref{eqn:SolK-intro} on each $K$-type $V_\delta$.
For the case of $\wSL(3,\R)$, as the maximal compact subgroup $\wK$ is isomorphic to 
$SU(2)$, such equations can be identified as
some recurrence relations that arise from the standard $\f{sl}(2)$-computation.
In this paper, instead of solving the recurrence relations, 
we realize each $\wK$-type as the space $\Pol_n[t]$ of polynomials of one variable
with degree $\leq n$,
in such a way that one can solve the equations in concern by solving 
ordinary differential equations such as Euler's hypergeometric equation.
In this realization it is revealed that there is 
a correspondence between the representations realized on the solution spaces
of $\D^\sigma_X$, $\D^\sigma_Y$, and $\D^\sigma_{\XcY}$,
and polynomial solutions to ordinary differential equations. 
For instance, as shown in Theorem \ref{thm:XcY},
three irreducible representations
are realized on the solution space of $\Cal{D}_{\XcY}^\sigma$ with 
$\sigma = \pp, \mathbb{H}, \mm$.
We denote these representations by
$\Pi_{(0)}$, $\Pi_{(\frac{1}{2})}$, and $\Pi_{(1)}$,
where $\Pi_{(\frac{a}{2})}$ denotes 
the irreducible representation with 
lowest $\wK$-type $V_{(\frac{a}{2})}$.
(For instance, the representation $\Pi_{(\frac{1}{2})}$ is Torasso's representation.)
Then Theorem \ref{thm:XcY}, Equation \eqref{eqn:Soln},
and Propositions \ref{prop:SolXcY} and \ref{prop:Mrep-wrho}
imply that
the representations $\Pi_{(0)}, \Pi_{(\frac{1}{2})}$, and $\Pi_{(1)}$
correspond to the following subspaces of $\Pol_n[t]$
with appropriate non-negative integer $n \in \Z_{\geq 0}$:
\begin{alignat*}{3}
&\Pi_{(0)}                   &&\longleftrightarrow&&\;\; \C u_n(t);\\
&\Pi_{(\frac{1}{2})} &&\longleftrightarrow &&\;\; \C u_n(t) \oplus \C v_n(t);\\ 
&\Pi_{(1)}               &&\longleftrightarrow &&\;\; \C v_n(t),
\end{alignat*} 
where $u_n(t)$ and $v_n(t)$ are given by
\begin{equation*}
u_n(t)= {}_2F_1[-\frac{n}{4}, -\frac{n-1}{4}, \frac{3}{4};t^4]
\quad \text{and} \quad
v_n(t)= t{}_2F_1[-\frac{n-1}{4}, -\frac{n-2}{4}, \frac{5}{4};t^4].
\end{equation*}
\noindent
The functions $u_n(t)$ and $v_n(t)$ form a fundamental set of solutions to 
Euler's hypergeometric equation 
$D{[-\frac{n}{4}, -\frac{n-1}{4}, \frac{3}{4};t^4]}f(t)=0$
with
$D{[a,b,c; x]}=x(1-x)\frac{d^2}{dx^2} + (c-(a+b+1)x)\frac{d}{dx}- ab$.
We also have a similar correspondence for the $\rho$ case 
(see 
Theorem \ref{thm:Ktype-rho}, 
Propositions \ref{prop:SolXandY} and \ref{prop:Mrep-rho},
and 
Corollary \ref{cor:SolXY}).

It is known that the representations 
$\Pi_{(0)}, \Pi_{(\frac{1}{2})}$, and $\Pi_{(1)}$ are all unitary.
We hope that we can also  report on an explicit construction of the unitary structures in future.
It is remarked that 
there exist two irreducible $(\fg, \wK)$-modules with lowest $\wK$-type $V_{(\frac{3}{2})}$,
but these are not unitary (see, for instance, \cite[Ex.~12.4]{Vogan91}).


We now outline the rest of this paper. 
This paper consists of seven sections with this introduction.
First we discuss intertwining differential operators and 
the Peter--Weyl theorem in Section \ref{sec:IDO}.
In this section we start by recalling the duality theorem between 
intertwining differential operators and homomorphisms between 
generalized Verma modules. 
We then study the space of $K$-finite solutions to an intertwining differential operator
via the Peter--Weyl theorem. The main results of this section 
are Theorems \ref{prop:SVW} and \ref{thm:Ksol}, which give Peter--Weyl type formulas
of the space of $K$-finite solutions. The common solution space of 
a system of intertwining differential operators is also discussed in this section.
In the end we illustrate as a recipe our technique to  
determine the $K$-type formulafor the case of a line bundle to a line bundle 
(the LL case).


In Section \ref{sec:Verma},
to prepare for the later application to $\wSL(3,\R)$,
we specialize $G$ to be split real and take parabolic $P$ to be a minimal parabolic $B$.
The purpose of this section is to discuss the classification and construction of
homomorphisms between Verma modules as a $(\fg,B)$-module.
We give a summary of our technique as a recipe at the end of this section.


Section \ref{sec:SL3} is a preliminary section for Sections \ref{sec:rho} and \ref{sec:wrho}.
In this section $G$ and $P$ are taken to be $G=\wSL(3,\R)$ and 
$P=\wB$, a minimal parabolic of $\wSL(3,\R)$, and
we settle necessary notation and normalizations for the later sections.
In particular we identify the elements of $\wM$ for $\wB = \wM AN$ 
with these of the corresponding linear group $M \subset SL(3,\R)$ in a canonical way.
We also recall the realization of 
irreducible representations of $\wK\simeq SU(2)$ 
as the space of polynomials of one variable. 


In Sections \ref{sec:rho} and \ref{sec:wrho},
by using the results from the previous sections,
we study the $\wK$-type formulas of the spaces of $\wK$-finite solutions to several
intertwining differential operators.
In Section \ref{sec:rho}, we consider the principal series representation 
with infinitesimal character $\rho$ and study the solution spaces of 
first-order differential operators $\Cal{D}^\sigma_X$ and $\Cal{D}^\sigma_Y$.
The $\wK$-type formulas are achieved in Theorem \ref{thm:XandY},
which shows Theorem \ref{thm:Ktype-rho}.
In Section \ref{sec:wrho}, we take the infinitesimal character to be $\wrho$ 
and consider second-order operator $\Cal{D}^\sigma_{\XcY}$. 
We accomplish the $\wK$-type formulas in 
Theorem \ref{thm:XcY1}, which concludes Theorem \ref{thm:XcY}.
In this section 
there is one proposition whose proof involves some classical 
facts on the Gauss hypergeometric series ${}_2F_1[a,b,c;x]$. 
We give the proof in Section \ref{sec:appendix} 
after recalling these facts.


\section{Intertwining differential operators and the Peter--Weyl theorem}
\label{sec:IDO}

The aim of this section is to discuss
intertwining differential operators
between degenerate principal series representations.
More precisely, we first review a well-known duality theorem between
intertwining differential operators and homomorphisms between generalized Verma modules.
The space of $K$-finite solutions to such differential operators
are then studied via the Peter--Weyl theorem.

\subsection{Duality theorem 
between degenerate principal series representations and generalized Verma modules}
\label{sec:duality}

We start by reviewing a well-known duality theorem between the space of 
intertwining differential operators between
degenerate principal series representations and that of
homomorphisms between generalized Verma modules.

Let $G$ be a real reductive Lie group 
with Lie algebra $\mathfrak{g}_0$. 
Choose a Cartan involution $\theta$ on $\fg_0$ and write
$\mathfrak{g}_0=\fk_0 \oplus \mathfrak{s}_0$ for the corresponding
Cartan decomposition with $\mathfrak{k}_0$ the $+1$ eigenspace
and $\mathfrak{s}_0$ the $-1$ eigenspace of $\theta$.
Let $\mathfrak{a}_0^{\min} \subset \fs_0$
be a maximal abelian subspace of $\fs_0$. 
Put $\fh_0:=\ft^{\min}_0 \oplus \fa^{\min}_{0}$,
where $\ft^{\min}_0$ is a maximal abelian subspace of 
$\fm^{\min}_0:=Z_{\fk_0}(\fa^{\min}_0)$.

For real Lie algebra $\mathfrak{y}_0$, we express its complexification by 
$\mathfrak{y}$ (simply omitting the subscript $0$)
and the universal enveloping algebra of $\fy$ by $\mathcal{U}(\fy)$.
For instance, $\fg$, $\fh$, and $\fa^{\min}$ denote the complexifications of $\fg_0$, $\fh_0$, and $\fa^{\min}_0$, respectively. 
We write $\gD \equiv \gD(\fg,\fh)$ for the set of roots of $\fg$ with respect to
the Cartan subalgebra $\fh$ and 
$\gS \equiv\gS(\fg_0, \fa_0^{\min})$ for that of restricted roots of $\fg_0$ 
with respect to $\fa_0^{\min}$.
Choose positive systems $\gD^+$ and $\gS^+$of $\gD$ and $\gS$, respectively,
such that $\gD^+$ and $\gS^+$ are compatible.
We write $\rho$ for half the sum of the positive roots $\alpha \in \gD^+$.

Let $\fp^{\min}_0$ be the minimal parabolic subalgebra of $\fg_0$ with 
Langlands decomposition 
$\fp^{\min}_0 = \fm^{\min}_0 \oplus \fa^{\min}_0 \oplus \fn^{\min}_0$,
where the nilpotent radical $\fn^{\min}_0$ corresponds to $\gS^+$.
Fix a standard parabolic subalgebra $\fp_0 \supset \fp_0^{\min}$ with
Langlands decomposition $\fp_0 = \fm_0 \oplus \fa_0 \oplus \fn_0$.
Let $P$ be a parabolic subgroup of $G$ with Lie algebra $\frak{p}_0$.
We write $P=MAN$ for the Langlands decomposition of $P$
corresponding to $\fp_0 = \fm_0 \oplus \fa_0 \oplus \fn_0$.

For $\mu \in \fa^* \simeq \Hom_\R(\fa_0,\C)$,
we define a one-dimensional representation $\C_{\mu}$ of $A$ by 
$a \mapsto e^{\mu}(a):= e^{\mu(\log a)}$ for $a\in A$.
Then, given a finite-dimensional  
representation 
$W_\sigma = (\sigma, W)$ of $M$ and weight $\lambda \in \fa^*$,
we define an $MA$-representation $W_{\sigma,\lambda}$ by
\begin{equation*}
W_{\sigma,\lambda}:=W_\sigma \otimes \C_{\lambda+\rho}.
\end{equation*}
As usual, by letting $N$ act on $W_{\sigma,\lambda}$ trivially, we regard
$W_{\sigma,\lambda}$ as a representation of $P$.
We then write $\mathcal{W}_{\sigma,\lambda} = G\times_{P}W_{\sigma,\lambda}$
for the $G$-equivariant homogeneous vector bundle over $G/P$ with
fiber $W_{\sigma,\lambda}$.
Then we form a degenerate principal series representation
\begin{equation*}
I_P(\sigma, \lambda) 
:= \text{Ind}_P^G(\sigma\otimes(\lambda+\rho)\otimes \mathbb{1})
\end{equation*}
of $G$ on the Fr\' echet space $C^\infty\left(G/P, \mathcal{W}_{\sigma,\lambda}\right)$ 
of smooth sections.

For the $P$-representation $W_{\sigma,\lambda}$, we set 
\begin{equation*}
W^\ssv_{\sigma, \lambda}:=W^\ssv_\sigma \otimes \C_{-(\lambda+\rho)},
\end{equation*}
where $W^\ssv_{\sigma}= (\sigma^\ssv, W^\ssv)$ 
is the contragredient representation of $W_{\sigma}$.
As $W_{\sigma,\lambda}^\ssv$ is a representation of $P$, it can be thought of as
a $\mathcal{U}(\fp)$-module.
We then define a 
$(\fg, P)$-module 
$M_\fp\left(\sigma^\ssv, -\lambda\right)$ 
(\emph{generalized Verma module})
induced from $W_{\sigma,\lambda}^\ssv$ by
\begin{equation*}
M_\fp\left(\sigma^\ssv, -\lambda\right):=\mathcal{U}(\fg)\otimes_{\mathcal{U}(\fp)}
W_{\sigma,\lambda}^\ssv.
\end{equation*}
The parabolic subgroup $P$ acts on 
$M_\fp\left(\sigma^\ssv, -\lambda\right)$ diagonally via the adjoint action $\Ad$ 
on $\mathcal{U}(\fg)$ and the representation 
$\sigma^\ssv \otimes e^{-(\lambda+\rho)} \otimes \mathbb{1}$ 
on $W^\ssv_{\sigma,\lambda}$.

Given a finite-dimensional  
representation $E_\eta \equiv (\eta, E)$ 
of $M$ and a weight $\nu \in \fa^*$,
we similarly define $P$-representations $E_{\eta,\nu}$ and
$E^\ssv_{\eta,\nu}$.
Let $\mathcal{E}_{\eta,\nu} \to G/P$ be the $G$-equivariant 
homogeneous vector bundle over $G/P$ with fiber $E_{\eta,\nu}$.
We realize the degenerate principal series representation $I_P(\eta, \nu)$
of $G$ on $C^\infty\left(G/P, \mathcal{E}_{\eta, \nu}\right)$.
We denote by 
$\mathrm{Hom}_G(I_P(\sigma, \lambda), I_P(\eta,\nu))$ 
the space of intertwining operators from
$I_P(\sigma, \lambda)$ to $I_P(\eta, \nu)$.
Then we set
\begin{equation*}
\mathrm{Diff}_{G}
(I_P(\sigma, \lambda), I_P(\eta,\nu)):=
\mathrm{Diff}
(I_P(\sigma, \lambda), I_P(\eta,\nu)) 
\cap
\mathrm{Hom}_G
(I_P(\sigma, \lambda), I_P(\eta,\nu)),
\end{equation*}
where 
$\mathrm{Diff}
(I_P(\sigma, \lambda), I_P(\eta,\nu)) $
denotes the space of differential operators from 
$I_P(\sigma, \lambda)$ to $I_P(\eta, \nu)$.

Let $R(X)$ denote the infinitesimal right translation of $X \in \fg_0$.
We extend it complex linearly to $\fg$ and naturally to $\mathcal{U}(\fg)$.
For finite-dimensional vector space $V$, we define
\begin{equation*}
\IP{\cdot}{\cdot}_V \colon V \times V^\ssv \to \C
\end{equation*}
as the natural pairing of $V$ and $V^\ssv$.
The following duality theorem plays a key role for our construction of 
intertwining differential operators.
For the proof see, for instance,
\cite[Lem.~2.4]{CS90},
\cite[Thm.~2.9]{KP1}, or
\cite[Prop.~1.2]{KR00}.

\begin{thm}\label{thm:duality}
\emph{(duality theorem)}
There exists a natural linear isomorphism
\begin{equation}\label{eqn:dualityP}
\mathcal{D}_{H\to D}\colon
\Hom_{P}\left(E_{\eta,\nu}^\ssv, 
M_\fp\left(\sigma^\ssv, -\lambda\right)  \right)
\stackrel{\sim}{\To}
\mathrm{Diff}_{G}
(I_P(\sigma, \lambda), I_P(\eta,\nu)).
\end{equation}
Equivalently, we have
\begin{equation}\label{eqn:duality}
\Hom_{\mathfrak{g}, P}\left(
M_\fp\left(\eta^\ssv, -\nu\right),
M_\fp\left(\sigma^\ssv, -\lambda\right)\right)
\stackrel{\sim}{\To}
\mathrm{Diff}_{G}
(I_P(\sigma, \lambda), I_P(\eta,\nu)).
\end{equation}
Moreover, for $\varphi \in 
\Hom_{P}\left(E_{\eta,\nu}^\ssv, 
M_\fp\left(\sigma^\ssv, -\lambda\right)  \right)$
with $\varphi(x^\ssv\otimes \mathbb{1}_{-(\nu+\rho)})
= \sum_{i}u_{x,i} \otimes (w_{x,i}^\ssv\otimes \mathbb{1}_{-(\lambda+\rho)})$ 
for 
$u_{x,i} \in \mathcal{U}(\fg)$ and 
$w_{x,i}^\ssv\otimes\mathbb{1}_{-(\lambda+\rho)} \in W^\ssv_{\sigma,\lambda}$,
the map $\mathcal{D}_{H\to D}$ is given by
\begin{equation}\label{eqn:duality2}
\langle \mathcal{D}_{H\to D}(\varphi)F, x^\ssv\rangle_{\tiny E} 
= \sum_{i} \langle R(u_{x,i})F, w_{x,i}^\ssv \rangle_{\tiny W}
\quad \text{for $F \in I_P(\sigma,\lambda)$}.
\end{equation}
\end{thm}

\subsection{Tensored operator $\D \otimes \mathrm{id}_\xi$}
\label{subsec:TOP}

For some representations $\xi$ of $M$, 
a differential operator $\Cal{D}  
\in \mathrm{Diff}_{G}(I_P(\sigma, \lambda), I_P(\eta,\nu))$
induces another operator 
$\Cal{D} \otimes \mathrm{id}_{\xi} \in 
\mathrm{Diff}_{G}(I_P(\sigma\otimes \xi, \lambda), I_P(\eta \otimes \xi,\nu))$.
To describe it more carefully, 
let $\Rep(M)_{\fin}$ denote the set of 
finite-dimensional representations of $M$.
As $M$ is not connected in general,
we write $M_0$ for the identity component of $M$.
Let $\Rep(M/M_0)$ denote the set of 
representations of  the component group $M/M_0$.
Via the surjection $M \twoheadrightarrow M/M_0$, 
we regard $\Rep(M/M_0)$ as a subset of $\Rep(M)_{\fin}$.

For $M_\fp\left(\sigma^\ssv, -\lambda\right)$, 
we write
$M_\fp\left(\sigma^\ssv, -\lambda\right)^\fn$
for the subspace of $\fn$-invariant elements of
$M_\fp\left(\sigma^\ssv, -\lambda\right)$, namely,
\begin{equation}\label{eqn:n-inv}
M_\fp\left(\sigma^\ssv, -\lambda\right)^\fn
:=\{v\in 
M_\fp\left(\sigma^\ssv, -\lambda\right) : 
X \cdot v = 0 
\text{ for all $X \in \fn$}\}.
\end{equation}
As $MA$ normalizes $\fn$, 
the subspace $M_\fp\left(\sigma^\ssv, -\lambda\right)^\fn$ is 
an $MA$-subrepresentation of $M_\fp\left(\sigma^\ssv, -\lambda\right)$.
Then, for $M_\fp(\eta^\ssv, -\nu)=\Cal{U}(\fg)\otimes_{\Cal{U}(\fp)}E^\ssv_{\eta,\nu}$,
we have
\begin{equation}\label{eqn:n}
\begin{aligned}
\Hom_{\mathfrak{g}, P}\left(
M_\fp\left(\eta^\ssv, -\nu\right),
M_\fp\left(\sigma^\ssv, -\lambda\right)\right)
&\simeq
\Hom_{P}\left(E^\ssv_{\eta,\nu}, 
M_\fp\left(\sigma^\ssv, -\lambda\right)\right) \\
&=
\Hom_{MA}\left(E^\ssv_{\eta,\nu}, 
M_\fp\left(\sigma^\ssv, -\lambda\right)^\fn\right).
\end{aligned}
\end{equation}

Now, for $U_\xi :=(\xi, U) \in \Rep(M/M_0)$,
we define a linear isomorphism
\begin{equation*}
\psi_{(\sigma^\ssv,-\lambda;\;\xi^\ssv)}\colon
W^{\ssv}_{\sigma,\lambda} \otimes U^\ssv_\xi
\stackrel{\sim}{\To}
(W^{\ssv}_{\sigma}\otimes U^\ssv_{\xi}) \otimes \C_{-(\lambda+\rho)}
\end{equation*}
simply by
\begin{equation*}
(w^\ssv\otimes  \mathbb{1}_{-(\lambda+\rho)}) \otimes z^\ssv
\mapsto 
(w^\ssv\otimes z^\ssv) \otimes \mathbb{1}_{-(\lambda+\rho)}.
\end{equation*}
We set
\begin{equation*}
\widetilde{\psi}_{(\sigma^\ssv,-\lambda;\;\xi^\ssv)}:=
\mathrm{id}_{\Cal{U(\fg)}}\otimes \psi_{(\sigma^\ssv,-\lambda;\;\xi^\ssv)},
\end{equation*}
where $\mathrm{id}_{\Cal{U(\fg)}}$ is the identity map on $\Cal{U(\fg)}$.
By regarding $U^\ssv_{\xi}$ as a $P=MAN$-representation with trivial $AN$-action,
the map $\widetilde{\psi}_{(\sigma^\ssv,-\lambda;\;\xi^\ssv)}$ defines  
an $MA$-isomorphism
\begin{equation}\label{eqn:psi}
\widetilde{\psi}_{(\sigma^\ssv,-\lambda;\;\xi^\ssv)}\colon
M_\fp(\sigma^\ssv,-\lambda)\otimes U^\ssv_\xi
\stackrel{\sim}{\To}
M_\fp(\sigma^\ssv\otimes \xi^\ssv,-\lambda).
\end{equation}
Then, for $\varphi \in \Hom_{MA}\left(E^\ssv_{\eta,\nu}, 
M_\fp\left(\sigma^\ssv, -\lambda\right)^\fn\right)$, 
we define 
an $MA$-homomorphism
\begin{equation*}
\xvarphi \colon
(E^\ssv_{\eta} \otimes U^\ssv_\xi)\otimes \C_{-(\nu+\rho)}
\To
M_\fp(\sigma^\ssv\otimes \xi^\ssv,-\lambda)
\end{equation*}
by
\begin{equation}\label{eqn:wphi}
\xvarphi= \widetilde{\psi}_{(\sigma^\ssv,-\lambda;\;\xi^\ssv)}
\circ (\varphi \otimes \mathrm{id_{\xi^{\ssv}}})
\circ\psi_{(\eta^\ssv,-\nu;\;\xi^\ssv)}^{-1},
\end{equation}
where $\mathrm{id}_{\xi^\ssv}$ is the identity map on $U^\ssv_\xi$.

\begin{lem}\label{lem:tensor}
Let $ (\eta, \nu), (\sigma, \lambda) \in \Rep(M)_\fin\times \fa^*$ and
$\varphi \in \Hom_{MA}\left(E^\ssv_{\eta,\nu}, 
M_\fp\left(\sigma^\ssv, -\lambda\right)^\fn\right)$.
Then, for $\xi \in \Rep(M/M_0)$, 
we have
\begin{equation}\label{eqn:xphi}
\xvarphi \in
\Hom_{MA}\left(
(E^\ssv_{\eta} \otimes U^\ssv_\xi)\otimes \C_{-(\nu+\rho)},
M_\fp(\sigma^\ssv\otimes \xi^\ssv,-\lambda)^\fn\right).
\end{equation}
\end{lem}

\begin{proof}
Let $U_\xi =(\xi, U) \in \Rep(M/M_0)$.
Observe that since $\fp$ acts on $U^\ssv_{\xi}$ trivially,
the linear isomorphism $\widetilde{\psi}_{(\sigma^\ssv,-\lambda;\;\xi^\ssv)}$ 
in \eqref{eqn:psi} induces an $MA$-isomorphism
\begin{equation}\label{eqn:tensor2}
M_\fp(\sigma^\ssv,-\lambda)^\fn\otimes U^\ssv_\xi
\simeq
M_\fp(\sigma^\ssv\otimes \xi^\ssv,-\lambda)^\fn.
\end{equation}
Since $\Im (\varphi \otimes \mathrm{id}_{\xi^\ssv}) \subset 
M_\fp(\sigma^\ssv,-\lambda)^\fn\otimes U^\ssv_\xi$ 
as $\Im \varphi \subset M_\fp(\sigma^\ssv,-\lambda)^\fn$,
the assertion simply follows from the definition \eqref{eqn:wphi} 
of $\xvarphi$.
\end{proof}

Via \eqref{eqn:n}, Lemma \ref{lem:tensor}
implies the following.

\begin{cor}\label{cor:tensor}
The following are equivalent on 
$(\sigma, \eta; \lambda, \nu) \in \Rep(M)_\fin^2\times (\fa^*)^2$.
\begin{enumerate}
\item[(i)]
$\Hom_{\mathfrak{g}, P}\left(
M_\fp\left(\eta^\ssv, -\nu\right),
M_\fp\left(\sigma^\ssv, -\lambda\right)\right) \neq \{0\}$.
\item[(ii)] 
$\Hom_{\mathfrak{g}, P}\left(
M_\fp\left(\eta^\ssv\otimes \xi^\ssv, -\nu\right),
M_\fp\left(\sigma^\ssv\otimes \xi^\ssv, -\lambda\right)\right)\neq \{0\}$ 
for any $\xi \in \Rep(M/M_0)$.
\end{enumerate}
\end{cor}

It follows from the duality theorem that
the intertwining differential operator 
$\mathcal{D}_{H\to D}(\xvarphi)$
corresponding to the homomorphism $\xvarphi$ in \eqref{eqn:xphi} is given by
\begin{equation}\label{eqn:dxphi}
\mathcal{D}_{H\to D}(\xvarphi) = 
\mathcal{D}_{H\to D}(\varphi)\otimes \mathrm{id}_\xi.
\end{equation}
The differential-operator counterpart of Lemma \ref{lem:tensor}
and Corollary \ref{cor:tensor} is then given as follows.

\begin{lem}\label{lem:tensor2}
Let $(\eta, \nu), (\sigma, \lambda) \in \Rep(M)_\fin\times \fa^*$
and $\D \in \mathrm{Diff}_{G}(I_P(\sigma, \lambda), I_P(\eta,\nu))$.
Then, for $\xi \in \Rep(M/M_0)$, we have
\begin{equation*}
\D\otimes \mathrm{id}_\xi \in
\mathrm{Diff}_{G}
(I_P(\sigma\otimes \xi, \lambda), I_P(\eta\otimes \xi,\nu)).
\end{equation*}
In particular, the following are equivalent on 
$(\sigma, \eta; \lambda, \nu) \in \Rep(M)_\fin^2\times (\fa^*)^2$.
\begin{enumerate}
\item[(i)]
$\mathrm{Diff}_{G}(I_P(\sigma, \lambda), I_P(\eta,\nu)) \neq \{0\}$.
\item[(ii)] 
$\mathrm{Diff}_{G}
(I_P(\sigma\otimes \xi, \lambda), I_P(\eta\otimes \xi,\nu))\neq \{0\}$ 
for any $\xi \in \Rep(M/M_0)$.
\end{enumerate}
\end{lem}

\subsection{Peter--Weyl type formula for the space of $K$-finite solutions}
\label{subsec:K-finite}

Let $K$ be the maximal compact subgroup of $G$ with Lie algebra $\fk_0$.
We next study the $K$-type decomposition of the space of $K$-finite solutions to 
intertwining differential operators.
The following lemma plays a key role.

\begin{lem}\label{lem:GK}
There exists a $(\Cal{U}(\fk \cap \fm), K\cap M)$-isomorphism
\begin{align}
\iota\colon
\mathcal{U}(\fk)\otimes_{\mathcal{U}(\fk \cap \fm)}  W^\ssv_{\sigma}
\stackrel{\sim}{\To}
\mathcal{U}(\fg) \otimes_{\mathcal{U}(\fp)} W^\ssv_{\sigma,\lambda},
\quad u\otimes w \mapsto u\otimes (w\otimes \mathbb{1}_{-(\lambda+\rho)}), \label{eqn:GK}
\end{align}
where $K\cap M$ acts on 
$\mathcal{U}(\fk)\otimes_{\mathcal{U}(\fk \cap \fm)}  W^\ssv_{\sigma}$
and 
$\mathcal{U}(\fg) \otimes_{\mathcal{U}(\fp)} W^\ssv_{\sigma,\lambda}$ diagonally.
\end{lem}

\begin{proof}
This is an immediate generalization of  \cite[Lem.~2.1]{Kable11}.
\end{proof}

We identify 
$C^\infty\left(G/P, \mathcal{W}_{\sigma, \lambda}\right)$
with the space  
$C^\infty(K/(K\cap M), \Cal{W}_{\sigma,\lambda}\vert_{K})$
of smooth sections for the restricted vector bundle
$ \Cal{W}_{\sigma,\lambda}\vert_{K}\to K/(K\cap M)$, 
which is further identified as
\begin{equation*}
C^\infty\left(K/(K\cap M), \Cal{W}_{\sigma,\lambda}\vert_{K}\right)
\simeq \left\{F\in C^\infty(K)\otimes W_{\sigma}: 
F(km) = \sigma^{-1}(m)F(k)\; 
\text{for all $m \in K\cap M$}
\right\}.
\end{equation*}

Let $\Irr(K)$ denote the set of equivalence classes of irreducible
representations of $K$. 
For $V_\delta = (\delta, V) \in \Irr(K)$,
we denote by $(\cdot, \cdot)_\delta$ a $K$-invariant Hermitian inner product of 
$V_\delta$. We take the first argument of $(\cdot, \cdot)_\delta$ to be linear 
and the second to be conjugate linear.
We then identify the contragredient representation $(\delta^\ssv, V^\ssv_{\delta})$
with the conjugate representation $(\bar{\delta}, \bar{V}_{\delta})$ via the map
$\bar{v} \mapsto (\cdot, v)_{\delta}$.

Let $I_P(\sigma,\lambda)_K$ denote the $(\fg, K)$-module
consisting of the $K$-finite vectors of 
$C^\infty(K/(K\cap M), \Cal{W}_{\sigma,\lambda}\vert_{K})$.
It then follows from the Peter--Weyl theorem that,
via the identification $V^\ssv_\delta \simeq \bar{V}_\delta$,
there exists a $K$-isomorphism 
\begin{equation}\label{eqn:PW}
\Phi_{\sigma}\colon \bigoplus_{\delta \in \Irr(K)}
V_{\delta} \otimes \left(\bar{V}_{\delta} \otimes W_\sigma \right)^{K\cap M}
\stackrel{\sim}{\To} 
I_P(\sigma, \lambda)_K
\end{equation}
\noindent
with 
$\Phi_{\sigma}:=\bigoplus_{\delta \in \Irr(K)}\Phi_{(\sigma, \delta)}$,
where the map
\begin{equation*}
\Phi_{(\sigma,\delta)}\colon
V_\delta \otimes 
\left(\bar{V}_\delta \otimes W_\sigma \right)^{K\cap M}
\To
I_P(\sigma, \lambda)_K
\end{equation*}
is given by
\begin{equation}\label{eqn:Phi}
\Phi_{(\sigma, \delta)}(v \otimes \overline{v'} \otimes w)(k):=
(v, \delta(k)v')_{\delta}w
=(\delta(k^{-1})v, v')_{\delta}w.
\end{equation}

\noindent
We note that the direct sum on the left-hand side of \eqref{eqn:PW} is algebraic.

Let $\tau\colon \mathfrak{g} \to \mathfrak{g}$ be 
the conjugation of $\mathfrak{g}$ with respect to the 
real form $\mathfrak{g}_0$, that is,
$\tau(X+\sqrt{-1}Y) = X-\sqrt{-1}Y$ for $X,Y \in \fg_0$.
Given $\delta \in \Irr(K)$, 
let $d\delta$ denote
the infinitesimal representation of $\mathfrak{k}_0$.
As usual we extend $\tau$ and $d\delta$ 
to the universal enveloping algebras $\Cal{U}(\fg)$ and $\Cal{U}(\fk)$, respectively.
It then follows from \eqref{eqn:Phi} and the conjugate-linearity of the second argument
of $(\cdot, \cdot)_\delta$ that,
for $u \in \mathcal{U}(\mathfrak{k})$, we have
\begin{equation}\label{eqn:Rdiff}
R(u) 
\Phi_{(\sigma, \delta)}(v \otimes \overline{v'} \otimes w)(k)
=(\delta(k^{-1})v, d\delta(\tau(u))v')_\delta w.
\end{equation}
\vskip 0.3in

Now, for $\varphi \in \Hom_{P}\left(E_{\eta,\nu}^\ssv, 
M_\fp\left(\sigma^\ssv, -\lambda\right)  \right)$, we write
\begin{equation}\label{eqn:dphi}
\D_\varphi=\D_{H \to D}(\varphi),
\end{equation}
the intertwining differential operator corresponding to $\varphi$
under the isomorphism $\D_{H\to D}$ in \eqref{eqn:dualityP}. 
It follows from the duality theorem that 
any $\D \in \Diff_G(I_P(\sigma, \lambda),I_P(\eta,\nu))$
is of the form $\D=\D_\varphi$ for some
$\varphi \in \Hom_{P}\left(E_{\eta,\nu}^\ssv, 
M_\fp\left(\sigma^\ssv, -\lambda\right)  \right)$, that is,
\begin{equation*}
\Diff_G(I_P(\sigma, \lambda),I_P(\eta,\nu))
=\{\D_\varphi : \varphi \in \Hom_{P}\left(E_{\eta,\nu}^\ssv, 
M_\fp\left(\sigma^\ssv, -\lambda\right)  \right)\}.
\end{equation*}
Then, for $\D_\varphi\in\Diff_G(I_P(\sigma, \lambda),I_P(\eta,\nu))$, 
we set
\begin{alignat*}{3}
&\Cal{S}ol_{(\varphi; \lambda)}(\sigma)
&&:=\{F \in I_P(\sigma, \lambda) &&: \D_\varphi F = 0\},\\
&\Cal{S}ol_{(\varphi; \lambda)}(\sigma)_K
&&:=\{F \in I_P(\sigma, \lambda)_K &&: \D_\varphi F = 0\},
\end{alignat*}
the spaces of smooth solutions and $K$-finite solutions 
to $\D_\varphi$, respectively.

As $\D_\varphi$ is an intertwining operator, the solution space 
$\Cal{S}ol_{(\varphi; \lambda)}(\sigma)$ is a representation of $G$.
Similarly, 
the space $\Cal{S}ol_{(\varphi; \lambda)}(\sigma)_K$ of $K$-finite solutions
is a $(\fg, K)$-module.
We wish to understand $\Cal{S}ol_{(\varphi; \lambda)}(\sigma)_K$
via the Peter--Weyl theorem \eqref{eqn:PW}.
To do so, for $x^\ssv\otimes\mathbb{1}_{-(\nu+\rho)} \in E_{\eta,\nu}^\ssv$, 
let $\varphi(x^\ssv\otimes\mathbb{1}_{-(\nu+\rho)})=\sum_{i}u_{x,i} 
\otimes (w_{x,i}^\ssv\otimes \mathbb{1}_{-(\lambda+\rho)}) 
\in M_\fp\left(\sigma^\ssv, -\lambda\right)$.
For $u_{x,i} \in \Cal{U}(\fg)$, we write $u_{x,i}^\flat \in \Cal{U}(\fk)$ such that 
\begin{equation*}
u_{x,i}^\flat \otimes w^\ssv_{x,i}=\iota^{-1}(u_{x,i}
\otimes (w_{x,i}^\ssv\otimes \mathbb{1}_{-(\lambda+\rho)}) ),
\end{equation*}
where $\iota$ is the isomorphism in \eqref{eqn:GK}.
Then, for $T = \sum_{j}\bar{v}_j \otimes w_j \in \bar{V}_\delta \otimes W_\sigma$,
we set
\begin{equation}\label{eqn:DTx}
D_\varphi(T;x^\ssv) := \sum_{i,j}
\IP{w_j}{w_{x,i}^\ssv}_{\tiny W}\;d\delta(\tau(u_{x,i}^\flat))v_j.
\end{equation}

\begin{lem}\label{lem:DFx}
Let $T$ and $\varphi(x^\ssv\otimes\mathbb{1}_{-(\nu+\rho)})$ be as above.
Then, for $F(k):=\Phi_{(\sigma,\delta)}(v\otimes T)(k) \in I_P(\sigma, \lambda)_K$,
we have 
\begin{equation}\label{eqn:DFx}
\IP{\Cal{D}_\varphi F(k)}{x^\ssv}_{\tiny E}=
(\delta(k^{-1})v, D_\varphi(T;x^\ssv))_\delta.
\end{equation}
\end{lem}

\begin{proof}
This is a direct consequence of \eqref{eqn:duality2} and \eqref{eqn:Rdiff}.
\end{proof}

For $\varphi \in \Hom_{P}\left(E_{\eta,\nu}^\ssv, 
M_\fp\left(\sigma^\ssv, -\lambda\right)  \right)$
and $J_{\delta, \sigma} \subset \bar{V}_\delta \otimes W_\sigma$,
we define
\begin{equation}\label{eqn:Solphi}
\Sol_{(\varphi)}\left(J_{\delta, \sigma}\right)
:=\{T \in J_{\delta, \sigma} : 
D_\varphi(T;x^\ssv) = 0 \;\; \text{for all $x^\ssv \in E^\ssv_{\eta}$}\}.
\end{equation}
Observe that we have
\begin{equation}\label{eqn:Solphi3}
\Sol_{(\varphi)}\left((\bar{V}_\delta \otimes W_\sigma)^{K\cap M}\right)
=\Sol_{(\varphi)}\left(\bar{V}_\delta \otimes W_\sigma\right)^{K\cap M}.
\end{equation}
For convenience to the later arguments we choose
the expression $\Sol_{(\varphi)}\left(\bar{V}_\delta \otimes W_\sigma\right)^{K\cap M}$
in Theorem \ref{prop:SVW} below.

\begin{thm}[PW type formula for the VV case]
\label{prop:SVW}
Let $\D_\varphi \in \Diff_G(I_P(\sigma, \lambda),I_P(\eta,\nu))$.
Then the space of 
$\Cal{S}ol_{(\varphi; \lambda)}(\sigma)_K$ 
of $K$-finite solutions to 
$\D_\varphi$
has a $K$-type decomposition
\begin{equation}\label{eqn:0706}
\Cal{S}ol_{(\varphi; \lambda)}(\sigma)_K
\simeq \bigoplus_{\delta \in \Irr(K)}V_\delta \otimes 
\Sol_{(\varphi)}\left(\bar{V}_\delta \otimes W_\sigma\right)^{K\cap M}.
\end{equation}
\end{thm}

\begin{proof}
Since $\Cal{S}ol_{(\varphi; \lambda)}(\sigma)_K$ is 
a $K$-subrepresentation of $I_P(\sigma,\lambda)_K$,
there exists a subspace 
$S_{\delta,\sigma} \subset (\bar{V}_\delta \otimes W_\sigma)^{K\cap M}$
such that 
\begin{equation*}
\Cal{S}ol_{(\varphi; \lambda)}(\sigma)_K
\simeq \bigoplus_{\delta \in \Irr(K)}V_\delta \otimes S_{\delta,\sigma}.
\end{equation*}
Observe that $F \in I_P(\sigma,\lambda)_K$ 
is in $\Cal{S}ol_{(\varphi; \lambda)}(\sigma)_K$ if and only if 
$\IP{\Cal{D}_\varphi F}{x^\ssv}_E = 0$ for all $x^\ssv \in E_{\eta}^\ssv$.
It then follows from \eqref{eqn:DFx} that we have
$S_{\delta,\sigma} = 
\Sol_{(\varphi)}\left((\bar{V}_\delta \otimes W_\sigma)^{K\cap M}\right)$.
Now \eqref{eqn:Solphi3} concludes the theorem.
\end{proof}

\begin{rem}
In the case that  $\sigma$ is one-dimensional, the isomorphism \eqref{eqn:0706}
is obtained in \cite[Thm.~2.6]{Kable11}
in the framework of conformally invariant systems.
\end{rem}

\subsection{Specialization to the case $(\Cal{L}_{\chi,\lambda},\, \Cal{E}_{\eta,\nu})$}
\label{subsec:line}

We now investigate the $K$-isomorphism \eqref{eqn:0706} 
for the case that the vector bundle
$\Cal{W}_{\sigma,\lambda}$ is specialized to
a line bundle $\Cal{L}_{\chi,\lambda}$
with character $\chi$ of $M$.
In this case intertwining differential operators
$\D \in \Diff_G(I_P(\chi,\lambda), I_P(\eta, \nu))$ are of the form
$\D = \D_\varphi$ for some
$\varphi \in \Hom_{P}\left(E_{\eta,\nu}^\ssv, 
M_\fp\left(\chi^{-1}, -\lambda\right)  \right)$.

Let $\varphi \in \Hom_{P}\left(E_{\eta,\nu}^\ssv, 
M_\fp\left(\chi^{-1}, -\lambda\right)  \right)$ and take
$x^\ssv \otimes\mathbb{1}_{-(\nu+\rho)} \in E_{\eta,\nu}^\ssv$.
We have $\varphi(x^\ssv\otimes\mathbb{1}_{-(\nu+\rho)} ) = u_x\otimes 
(\mathbb{1}_{\chi^{-1}}\otimes \mathbb{1}_{-(\lambda+\rho)})$ 
for some $u_x \in \Cal{U}(\bar{\fn})$,
where $\bar{\fn}$ denotes the opposite nilpotent radical to $\fn$.
(Here we confuse the notation $\bar{\fn}$ with $\bar{V}_\delta$.)
It follows from \eqref{eqn:DTx} that, 
for $\bar{v} \otimes \mathbb{1}_\chi \in \bar{V}_{\delta}\otimes \C_\chi$, 
we have
\begin{equation*}
D_\varphi(\bar{v}\otimes \mathbb{1}_\chi;x^\ssv) = d\delta(\tau(u^\flat_x))v.
\end{equation*}
Then we simply write 
\begin{equation*}
D_\varphi(v; x^\ssv) = D_\varphi(\bar{v} \otimes \mathbb{1}_\chi; x^\ssv).
\end{equation*}
We set
\begin{equation}\label{eqn:Solphi2}
\Sol_{(\varphi)}(\delta):=
\{v \in V_\delta : D_\varphi(v;x^\ssv) = 0 \;\; \text{for all $x^\ssv \in E^\ssv_{\eta}$}\}.
\end{equation}

\begin{lem}
For $m \in K\cap M$, we have 
\begin{equation}\label{eqn:Dmv}
D_\varphi(\delta(m)v ; x^\ssv) = \delta(m)D_\varphi(v; \chi(m^{-1})\eta^\ssv(m^{-1})x^\ssv).
\end{equation}
\end{lem}

\begin{proof}
We have
\begin{align*}
D_\varphi(\delta(m)v;x^\ssv) 
=d\delta(\tau(u_x^\flat))\delta(m)v
=\delta(m)d\delta(\tau((\Ad(m^{-1})u_x)^\flat))v.
\end{align*}
On the other hand, as $\varphi \in \Hom_{P}\left(E_{\eta,\nu}^\ssv,  
M_\fp\left(\chi^{-1}, -\lambda\right)  \right)$, we have
\begin{align*}
\varphi(\eta^\ssv(m^{-1})x^\ssv)
=m^{-1} \cdot \varphi(x^\ssv)
=\chi(m) \Ad(m^{-1})u_x\otimes 
(\mathbb{1}_{\chi^{-1}}\otimes \mathbb{1}_{-(\lambda+\rho)}).
\end{align*}
Thus $D_\varphi(v; \chi(m^{-1})\eta^\ssv(m^{-1})x^\ssv)= d\delta(\tau((\Ad(m^{-1})u_x)^\flat))v$.
Now the lemma follows.
\end{proof}

\begin{lem}\label{lem:SolM}
The space 
$\Sol_{(\varphi)}(\delta)$ is a $K\cap M$-representation.
\end{lem}

\begin{proof}
Take $v \in \Sol_{(\varphi)}(\delta)$ and $m \in K\cap M$.
For $x^\ssv \in E^\ssv_{\eta}$, we have
$\chi(m^{-1})\Ad(m^{-1})x^\ssv \in E_{\eta}^\ssv$.
Thus
$D_\varphi(v; \chi(m^{-1})\eta^\ssv(m^{-1})x^\ssv)=0$
as $v \in \Sol_{(\varphi)}(\delta)$.
It then follows from 
\eqref{eqn:Dmv} that
\begin{equation*}
D_\varphi(\delta(m)v ; x) = \delta(m)D_\varphi(v; \chi(m^{-1})\eta^\ssv(m^{-1})x^\ssv)=0.
\end{equation*}
This concludes the lemma.
\end{proof}

\begin{prop}
[PW type formula for the LV case 1]
\label{prop:PWSol}
Let $\D_{\varphi}\in \mathrm{Diff}_{G}
(I_P(\chi, \lambda), I_P(\eta,\nu))$.
Then there exists a $K$-isomorphism
\begin{equation*}
\Cal{S}ol_{(\varphi; \lambda)}(\chi)_K
\simeq \bigoplus_{\delta \in \Irr(K)}V_\delta 
\otimes \Hom_{K\cap M}(\Sol_{(\varphi)}(\delta), \chi).
\end{equation*}
\end{prop}

\begin{proof}
We have
\begin{align*}
\Sol_{(\varphi)}\left(\bar{V}_\delta \otimes \C_{\chi}\right)^{K\cap M}
&=(\overline{\Sol_{(\varphi)}(\delta)}\otimes \C_\chi)^{K\cap M}\\
&\simeq(\Sol_{(\varphi)}(\delta)^\vee \otimes \C_\chi)^{K\cap M}\\
&\simeq 
\Hom_{K\cap M}(\Sol_{(\varphi)}(\delta), \chi).
\end{align*}
Now the proposition follows from Theorem \ref{prop:SVW}.
\end{proof}

\subsection{Solution space of the tensored operator $\D_\varphi\otimes \mathrm{id_\xi}$}
\label{subsec:tensor}

Now let the character $\chi$ of $M$ in Section \ref{subsec:line} 
be the trivial character $\chi= \chi_\triv$ and
take $U_\xi = (\xi, U) \in \Rep(M/M_0)$.
Let $\varphi \in \Hom_{MA}\left(E^\ssv_{\eta,\nu}, 
M_\fp\left(\chi_\triv, -\lambda\right)^\fn\right)$ so that
$\D_{\varphi}\in \mathrm{Diff}_{G}
(I_P(\chi_\triv, \lambda), I_P(\eta,\nu))$.

First observe that it follows from Lemma \ref{lem:tensor} that
\begin{equation*}
\xvarphi \in
\Hom_{MA}\left(
(E^\ssv_{\eta} \otimes U^\ssv_\xi)\otimes \C_{-(\nu+\rho)},
M_\fp\left(\xi^\ssv, -\lambda\right)^\fn\right).
\end{equation*}
Equivalently, by Lemma \ref{lem:tensor2}, we have
\begin{equation*}
\D_\varphi\otimes \mathrm{id}_\xi \in 
\mathrm{Diff}_{G}
(I_P(\xi, \lambda), I_P(\eta \otimes \xi,\nu)).
\end{equation*}
As $\D_{\xvarphi} = \D_{\varphi}\otimes \mathrm{id}_{\xi}$ by 
\eqref{eqn:dxphi} with \eqref{eqn:dphi}, the solution spaces
$\Cal{S}ol_{(\xvarphi; \lambda)}(\xi)$ and $\Cal{S}ol_{(\xvarphi; \lambda)}(\xi)_K$
of $\D_{\xvarphi}$ are given by
\begin{alignat*}{3}
&\Cal{S}ol_{(\xvarphi; \lambda)}(\xi)
&&=\{F \in I_P(\xi, \lambda) &&: (\D_\varphi\otimes \text{id}_\xi)F = 0\},\\
&\Cal{S}ol_{(\xvarphi; \lambda)}(\xi)_K
&&=\{F \in I_P(\xi, \lambda)_K &&: (\D_\varphi\otimes \text{id}_\xi)F = 0\}.
\end{alignat*}
It follows Theorem \ref{prop:SVW} that we have 
\begin{equation}\label{eqn:0715}
\Cal{S}ol_{(\varphi_\xi; \lambda)}(\xi)_K
\simeq \bigoplus_{\delta \in \Irr(K)}V_\delta \otimes 
\Sol_{(\xvarphi)}\left(\bar{V}_\delta \otimes U_\xi\right)^{K\cap M}.
\end{equation}

\begin{lem}\label{lem:Solxi}
For 
$\varphi \in \Hom_{MA}\left(E^\ssv_{\eta,\nu}, 
M_\fp\left(\chi_\triv, -\lambda\right)^\fn\right)$
and $U_\xi = (\xi, U) \in \Rep(M/M_0)$, we have
\begin{equation*}
\Sol_{(\xvarphi)}
\left(\bar{V}_\delta \otimes U_\xi\right)
= \overline{\Sol_{(\varphi)}(\delta)}\otimes U_\xi.
\end{equation*}
\end{lem}

\begin{proof}
First observe that
the space $\Sol_{(\xvarphi)}
\left(\bar{V}_\delta \otimes U_\xi\right)$ is given by
\begin{equation*}
\Sol_{(\xvarphi)}
\left(\bar{V}_\delta \otimes U_\xi\right)
=\{T \in \bar{V}_\delta \otimes U_\xi : D_{\xvarphi}(T; S)=0 
\;\; \text{for all $S \in E^\ssv_{\eta}\otimes U^\ssv_\xi$}\}
\end{equation*}
(see \eqref{eqn:Solphi}).
For $S=\sum_i x_i^\ssv \otimes z^\ssv_i \in E^\ssv_{\eta}\otimes U^\ssv_\xi$ 
and $\varphi(\sum_i x_i^\ssv \otimes \mathbb{1}_{-(\nu+\rho)})
=\sum_i u_{x,i} \otimes  \mathbb{1}_{-(\lambda+\rho)}$,
we have 
\begin{equation*}
\xvarphi(S\otimes \mathbb{1}_{-(\nu+\rho)})
=\sum_iu_{x,i} \otimes (z^\ssv_i\otimes \mathbb{1}_{-(\lambda+\rho)}).
\end{equation*}
It then follows from \eqref{eqn:DTx} that,
for $T = \sum_{j}\bar{v}_j \otimes z_j \in \bar{V}_\delta \otimes U_\xi$,
the value $D_{\xvarphi}(T; S)$ is obtained by
\begin{equation}\label{eqn:DTS}
D_{\xvarphi}(T;S) = \sum_{i,j}
\IP{z_j}{z_i^\ssv}_{\tiny U}\;d\delta(\tau(u_{x,i}^\flat))v_j.
\end{equation}
On the other hand, recall from \eqref{eqn:Solphi2} that we have
\begin{equation}
\Sol_{(\varphi)}(\delta)=
\{v \in V_\delta : D_\varphi(v;x^\ssv) = 0 \;\; \text{for all $x^\ssv \in E^\ssv_{\eta}$}\}
\end{equation}
with
\begin{equation}\label{eqn:0716}
D_\varphi(v; x^\ssv) = D_\varphi(\bar{v} \otimes \mathbb{1}_{\chi_\triv}; x^\ssv)
= d\delta(\tau(u^\flat_x))v. 
\end{equation}
Then the inclusion 
$\overline{\Sol_{(\varphi)}(\delta)}\otimes U_\xi \subset
\Sol_{(\xvarphi)}\left(\bar{V}_\delta \otimes U_\xi\right)$
follows from \eqref{eqn:DTS} and \eqref{eqn:0716}.

To show the other inclusion,
let $\{b_1,\ldots, b_n\}$ be a basis of $V_\delta$ with
$n = \dim_\C V_\delta$. Then, as $\{\bar{b}_1,\ldots, \bar{b}_n\}$ is a basis of 
$\overline{V}_\delta$, the element 
$T = \sum_{j}\bar{v}_j \otimes z_j $ can be expressed as 
$T=\sum_{\ell}\bar{b}_\ell\otimes z'_\ell$ for some $z'_\ell \in U_\xi$.
If $T\in \Sol_{(\xvarphi)}\left(\bar{V}_\delta \otimes U_\xi\right)$,
then, for any $x^\ssv \otimes z^\ssv \in E^\ssv \otimes U_\xi^\ssv$, we have
\begin{equation*}
D_{\xvarphi}(T;x^\ssv \otimes z^\ssv ) = \sum_{\ell}
\IP{z'_\ell}{z^\ssv}_{\tiny U}\;d\delta(\tau(u_{x}^\flat))b_\ell=0,
\end{equation*}
which implies
\begin{equation*}
\IP{z'_\ell}{z^\ssv}_{\tiny U}\;d\delta(\tau(u_{x}^\flat))b_\ell=0
\quad \text{for all $\ell=1, \ldots, n$}.
\end{equation*}
As $z^\ssv \in U_\xi^\ssv$ is arbitrary, this shows that 
$d\delta(\tau(u_{x}^\flat))b_\ell=0$ for all $\ell$.
Moreover, since $x^\ssv \in E^\ssv_\eta$ is also arbitrary,
we obtain $\bar{b}_\ell \in \overline{\Sol_{(\varphi)}(\delta)}$ for all 
$\ell$. Therefore, for any $x^\ssv \in E^\ssv_{\eta}$, we have
$d\delta(\tau(u_{x}^\flat))v_j=0$ for all $j$, that is,
$T \in \overline{\Sol_{(\varphi)}(\delta)}\otimes U_\xi$.
This completes the proof.
\end{proof}

Now  \eqref{eqn:0715} and Lemma \ref{lem:Solxi} conclude the following
theorem.

\begin{thm}
[PW type formula for the LV case 2]
\label{thm:Ksol}
Let $\D_{\varphi}\in \mathrm{Diff}_{G}
(I_P(\chi_\triv, \lambda), I_P(\eta,\nu))$
and $\xi \in \Rep(M/M_0)$.
Then the space
$\Cal{S}ol_{(\xvarphi; \lambda)}(\xi)_K$ 
of $K$-finite solutions to 
$\D_\varphi \otimes \mathrm{id}_{\xi}$
has a $K$-type decomposition
\begin{equation}\label{eqn:Ksol}
\Cal{S}ol_{(\xvarphi; \lambda)}(\xi)_K
\simeq \bigoplus_{\delta \in \Irr(K)}
V_\delta \otimes \mathrm{Hom}_{K\cap M}\left(
\Sol_{(\varphi)}(\delta), \xi\right).
\end{equation}
\end{thm}

\begin{rem}
Theorem \ref{thm:Ksol} can be thought of as a generalization of 
the argument given after the proof of Theorem 2.6 of  \cite{Kable11}.
\end{rem}

If $G$ is split real and 
$P=MAN$ is minimal parabolic, then $M=M/M_0$.
Therefore, for $\D_{\varphi}\in \mathrm{Diff}_{G}
(I_P(\chi_\triv, \lambda), I_P(\eta,\nu))$ and $\sigma\in \Rep(M)$, 
we have
$\D_\varphi\otimes \mathrm{id}_\sigma \in 
\mathrm{Diff}_{G}
(I_P(\sigma, \lambda), I_P(\eta \otimes \sigma,\nu))$.
In this case, as $M\subset K$, 
Theorem \ref{thm:Ksol} is given as follows.

\begin{cor}\label{cor:Frob3}
Suppose that $G$ is split real and 
$P=MAN$ is minimal parabolic.
Let $\D_{\varphi}\in \mathrm{Diff}_{G}
(I_P(\chi_\triv, \lambda), I_P(\eta,\nu))$
and  $\sigma \in \Irr(M)$.
Then the space 
$\Cal{S}ol_{(\varphi_{\sigma}; \lambda)}(\sigma)_K$
of $K$-finite solutions to 
$\D_{\varphi} \otimes \mathrm{id}_{\sigma}$
has a $K$-type decomposition
\begin{equation*}
\Cal{S}ol_{(\varphi_{\sigma}; \lambda)}(\sigma)_K
\simeq 
\bigoplus_{\delta \in \Irr(K)}
V_\delta \otimes 
\mathrm{Hom}_{M}\left(
\Sol_{(\varphi)}(\delta), \sigma\right).
\end{equation*}
\end{cor}

\subsection{Common solution space of a system of intertwining differential operators}
\label{subsec:sol}

In this subsection we discuss the common solution space of
a system of intertwining differential operators
$\D_{\varphi_j}\in \mathrm{Diff}_{G}
(I_P(\sigma, \lambda), I_P(\eta_j,\nu_j))$
with $(\sigma, \lambda), (\eta_j,\nu_j) \in \Rep(M)_\fin \times \fa^*$ for $j=1,\ldots, n$.

For each $\D_{\varphi_j}\in \mathrm{Diff}_{G}
(I_P(\sigma, \lambda), I_P(\eta_j,\nu_j))$, we have 
$\Cal{S}ol_{(\varphi_j; \lambda)}(\sigma) \subset 
I_P(\sigma, \lambda)$.
Thus one can consider the common solution spaces
\begin{alignat*}{3}
&\Cal{S}ol_{(\varphi_1, \ldots, \varphi_n; \lambda)}(\sigma)
&&:=
\bigcap_{j=1}^n
\Cal{S}ol_{(\varphi_j; \lambda)}(\sigma)
&&\subset I_P(\sigma, \lambda),\\
&\Cal{S}ol_{(\varphi_1, \ldots, \varphi_n; \lambda)}(\sigma)_K
&&:=
\bigcap_{j=1}^n
\Cal{S}ol_{(\varphi_j; \lambda)}(\sigma)_K
&&\subset I_P(\sigma, \lambda)_K.
\end{alignat*}
For $(\delta, V_\delta) \in \Irr(K)$, we set
\begin{equation*}
\Sol_{(\varphi_1, \ldots, \varphi_n)}\left(\bar{V}_\delta \otimes W_\sigma\right)
:=
\bigcap_{j=1}^n
\Sol_{(\varphi_j)}\left(\bar{V}_\delta \otimes W_\sigma\right),
\end{equation*}
where  
$\Sol_{(\varphi_j)}\left(\bar{V}_\delta \otimes W_\sigma\right)$ is the 
subspace of $\bar{V}_\delta \otimes W_\sigma$ 
defined as in \eqref{eqn:Solphi}. It then follows from Theorem \ref{prop:SVW} that
the space $\Cal{S}ol_{(\varphi_1, \ldots, \varphi_n; \lambda)}(\sigma)_K$ 
of $K$-finite solutions to the system of differential operators 
$\D_{\varphi_1}, \ldots, \D_{\varphi_n}$ 
can be decomposed as
\begin{equation*}
\Cal{S}ol_{(\varphi_1, \ldots, \varphi_n; \lambda)}(\sigma)_K
\simeq 
\bigoplus_{\delta \in \Irr(K)}
V_\delta \otimes
\Sol_{(\varphi_1, \ldots, \varphi_n)}\left(\bar{V}_\delta \otimes W_\sigma\right)^{K\cap M}.
\end{equation*}

Now take $\sigma$ to be the trivial character $\sigma = \chi_\triv$. 
It then follows from Lemma \ref{lem:tensor2} that,
for any $\xi\in \Rep(M/M_0)$, we have
\begin{equation*}
\D_{\varphi_j}\otimes \mathrm{id}_\xi \in 
\mathrm{Diff}_{G}
(I_P(\xi, \lambda), I_P(\eta_j\otimes \xi,\nu_j))
\quad \text{for $j=1,\ldots,n$}.
\end{equation*}
Therefore, for each $\xi \in \Rep(M/M_0)$ and $\delta \in \Irr(K)$, we set
\begin{alignat}{1}
\Cal{S}ol_{((\varphi_1,\ldots, \varphi_n)_\xi; \lambda)}(\xi)_K
&:=\bigcap_{j=1}^n
\Cal{S}ol_{((\varphi_j)_\xi; \lambda)}(\xi)_K,\nonumber\\
\Sol_{(\varphi_1,\ldots, \varphi_n)}(\delta)
&:=\bigcap_{i=1}^n \Sol_{(\varphi_j)}(\delta), \label{eqn:SysSol}
\end{alignat}
where $\Sol_{(\varphi_j)}(\delta)$ is the subspace 
of $V_\delta$ defined as in \eqref{eqn:Solphi2}.
It then follows from Theorem \ref{thm:Ksol} that, 
for $\xi \in \Rep(M/M_0)$, there exists a
$K$-isomorphism
\begin{equation*}
\Cal{S}ol_{((\varphi_1,\ldots, \varphi_n)_\xi; \lambda)}(\xi)_K
\simeq \bigoplus_{\delta \in \Irr(K)}
V_\delta \otimes \mathrm{Hom}_{K \cap M}\left(
\Sol_{(\varphi_1,\dots,\varphi_n)}(\delta), \xi\right).
\end{equation*}

\subsection{Recipe for determining the $K$-type formula
for the case $(\Cal{L}_{\chi_\triv, \lambda},\, \Cal{L}_{\chi,\nu})$}
\label{subsec:recipe1}

In this subsection, for the later applications in mind, 
we further take the targeted vector bundle $\Cal{E}_{\eta,\nu}$ 
to be a line bundle $\Cal{L}_{\chi,\nu}$ and 
summarize as a recipe
how to determine the $K$-type formula of 
the solution space of the
intertwining differential operator $\Cal{D}$ from 
$\Cal{L}_{\chi_\triv, \lambda}$ to $\Cal{L}_{\chi,\nu}$.

To the end, we first observe that in this case,
by the duality theorem and a general fact 
on the space of homomorphisms between generalized Verma modules
(\cite[Thm.~1.1]{Lepowsky76}),
we have $\dim_\C\Diff_G(I_P(\chi_\text{triv}, \lambda), I_P(\chi,\nu)) \leq 1$
for any character $\chi$ of $M$ and $\lambda, \nu \in\fa^*$.
Further, 
as any map $\varphi \in \Hom_P
\left(\C_{\chi^{-1}, -\nu}, M_\fp\left(\chi_\triv, -\lambda\right)  \right)$
is of the form
$\varphi(\mathbb{1}_{\chi^{-1}}\otimes \mathbb{1}_{-(\nu+\rho)})
= u \otimes (\mathbb{1}_{\chi_\triv} \otimes \mathbb{1}_{-(\lambda+\rho)})$
with $u \in \Cal{U}(\bar \fn)$,
the differential operator $\D \in \Diff_G(I_P(\chi_\triv, \lambda), I_P(\chi,\nu))$
is given by $\D = R(u)$ for some $u \in \Cal{U}(\bar\fn)$.
We then write $\D_u$ for the differential operator
on $I_P(\chi_\triv,\lambda)$ such that $\D_u = R(u)$.

To simplify the notation we write
$\Cal{S}ol_{(u;\lambda)}(\xi)$ 
for the solution space of $\D_u\otimes \mathrm{id}_{\xi}$.
We also write  
\begin{equation}\label{eqn:Sol}
\Sol_{(u)}(\delta)= \{v \in V_\delta :
d\delta(\tau( u^\flat) )v = 0 \}.
\end{equation}
With the notation the $K$-type decomposition \eqref{eqn:Ksol} is given by
\begin{equation}\label{eqn:Ksol2}
\Cal{S}ol_{(u; \lambda)}(\xi)_K
\simeq \bigoplus_{\delta \in \Irr(K)}
V_\delta \otimes \mathrm{Hom}_{K\cap M}\left(
\Sol_{(u)}(\delta), \xi\right).
\end{equation}
(See Theorem \ref{thm:intro1}.)

Let $\Irr(M/M_0)$ denote the set of equivalence classes of 
irreducible representations of $M/M_0$.
The aim of the recipe is to determine the representations
$\delta \in \Irr(K)$ and $\xi\in \Irr(M/M_0)$ such that
$\Hom_{K\cap M}(\Sol_{(u)}(\delta), \xi)\neq \{0\}$.
There are five steps in the recipe.

\vskip 0.1in
\noindent
\textsf{Recipe for determining the $K$-type decomposition}. 
Let $\D_u \in \Diff_G(I_P(\chi_\triv,\lambda), I_P(\chi,\nu))$.
The $K$-type formula for the solution space 
$\Cal{S}ol_{(u;\lambda)}(\xi)$ of $\D_{u}\otimes \mathrm{id}_\xi$ 
can be obtained as follows.
\vskip 0.1in

\noindent
\textsf{Step S1:}
Determine $\tau(u^\flat) \in \Cal{U}(\fk_0)$ for $u \in \Cal{U}(\bar\fn)$.
\vskip 0.2in

\noindent
\textsf{Step S2:}
Choose a realization of $\delta \in \Irr(K)$ to solve the equation
$d\delta(\tau(u^\flat) )v = 0$ (see \eqref{eqn:Sol}).
Find the explicit formula of the operator $d\delta(\tau(u^\flat)) \in \text{End}(V_\delta)$ 
if necessary. 
\vskip 0.2in

\noindent
\textsf{Step S3:}
Solve the equation $d\delta(\tau(u^\flat) )v = 0$
in the realization chosen in Step S2
and determine 
$\delta \in \Irr(K)$ such that $\Sol_{(u)}(\delta)\neq \{0\}$.
\vskip 0.2in

\noindent
\textsf{Step S4:}
For  $\delta \in \Irr(K)$ with $\Sol_{(u)}(\delta)\neq \{0\}$,
determine the $K\cap M$-representation on 
$\Sol_{(u)}(\delta)$.
\vskip 0.2in

\noindent
\textsf{Step S5:}
Given $\xi \in \Irr(M/M_0)$,
determine $\delta \in \Irr(K)$ with $\Sol_{(u)}(\delta)\neq \{0\}$ such that
\begin{equation*}
\mathrm{Hom}_{K\cap M}\left(
\Sol_{(u)}(\delta),\, \xi\right) \neq \{0\}.
\end{equation*}
We remark that through the recipe one can also determine
$\dim_\C \mathrm{Hom}_{K\cap M}\left(
\Sol_{(u)}(\delta),\, \xi\right)$.
(See Propositions \ref{prop:Krep-rho} and \ref{prop:Krep-wrho}.)
\vskip 0.2in

Given characters $\chi_j$ of $M$ and $\nu_j \in \fa^*$ for $j=1,\ldots, n$,
take differential operators
$\Cal{D}_{u_j} \in \Diff_G(I_P(\chi_\triv, \lambda), I_P(\chi_j, \nu_j))$.
Then the $K$-type formula for the common solution space 
$\Cal{S}ol_{(u_1, \ldots, u_n; \lambda)}(\xi)$
of the operators $\Cal{D}_{u_1}\otimes\mathrm{id}_\xi, \ldots, 
\Cal{D}_{u_n}\otimes \mathrm{id}_\xi$ 
can be achieved as follows.
\vskip 0.1in

\noindent
\textsf{Step CS1: }
Perform Steps S1--S3 for each $\D_{u_j}\otimes \mathrm{id}_\xi$.
\vskip 0.2in

\noindent
\textsf{Step CS2:}
Determine $\delta \in \Irr(K)$ such that
$\Sol_{(u_1,\ldots, u_n)}(\delta)\neq \{0\}$.
\vskip 0.2in

\noindent
The rest of the steps are the same as Steps S4 and S5 by replacing 
$\Sol_{(u)}(\delta)$ with $\Sol_{(u_1,\ldots, u_n)}(\delta)$.
\vskip 0.1in

In Sections \ref{sec:rho} and \ref{sec:wrho}, we determine the $K$-type decompositions
for certain intertwining differential operators in accordance with the recipe
for $G = \widetilde{SL}(3,\R)$.

\section{$(\fg, B)$-homomorphisms between Verma modules}\label{sec:Verma}

The duality theorem (Theorem \ref{thm:duality}) shows that 
the classification and construction of intertwining differential operators
between degenerate principal series representations are equivalent to those of 
$(\fg, P)$-homomorphisms between generalized Verma modules.
In this section, for later convenience, we discuss
$(\fg, B)$-homomorphisms between (full) Verma modules,
where $B$ is a minimal parabolic subgroup of a split real simple Lie group $G$.

\subsection{Notation}
\label{subsec:Verma0}

Let $G$ be a split real simple Lie group and 
fix a minimal parabolic subgroup $B=MAN$.
In this case, as $M$ is discrete, we have $\Irr(M)=\Irr(M)_\fin=\Irr(M/M_0)$.
We denote by $\fg, \fb, \fa$, and $\fn$ the complexifications of
the Lie algebras of $G, B, A$, and $N$, respectively.
Then $\fb = \fa \oplus \fn$ is a Borel subalgebra of $\fg$.
For $(\sigma, \lambda) \in \Irr(M) \times \fa^*$, we write 
\begin{equation*}
M(\sigma,\lambda) =
M_{\fb}(\sigma,\lambda)
:=\Cal{U}(\fg)\otimes_{\Cal{U}(\fb)}(\sigma \otimes (\lambda-\rho)),
\end{equation*}
a (full) Verma module.
By letting $B$ act diagonally,
we regard $M(\sigma,\lambda)$ as a $(\fg, B)$-module. 
When $M(\sigma,\lambda)$ is regarded just as a $\fg$-module, we write
$M(\sigma, \lambda)\vert_{\fg}$. 
When $\sigma$ is the trivial character $\chi_\triv$,
we also write 
\begin{equation*}
M(\lambda) = M(\chi_\triv, \lambda)\vert_{\fg}.
\end{equation*}

Let $\Irr(M)_{\mathrm{char}}$ denote the set of characters of $M$.
For any $\chi \in \Irr(M)_{\mathrm{char}}$, we have
$M(\chi, \lambda)\vert_{\fg}=M(\chi_\triv, \lambda)\vert_{\fg}$.
Thus $M(\chi, \lambda)\vert_{\fg}=M(\lambda)$
for $\chi \in \Irr(M)_{\mathrm{char}}$.

For $(\sigma, \lambda) \in \Irr(M) \times \fa^*$, we write 
\begin{equation*}
I(\sigma, \lambda) = I_{B}(\sigma, \lambda),
\end{equation*}
where $I_{B}(\sigma, \lambda)$ is 
the principal series representation
as in Section \ref{sec:IDO}.

By Corollary \ref{cor:Frob3},
we are initially interested in 
$\mathcal{D} \in \Diff_{G}\left(I(\chi_\triv, \lambda_1), I(\sigma, \lambda_2)\right)$
for $\sigma \in \Irr(M)$.
In Sections \ref{subsec:Verma1} and \ref{subsec:Verma2} below
we shall discuss the classification and constructions of 
the differential operators $\mathcal{D}$
in terms of homomorphisms between Verma modules.

\subsection{Classification of homomorphisms between Verma modules}
\label{subsec:Verma1}

We start with the classification of the parameters
$(\sigma, \lambda_1, \lambda_2) 
\in \Irr(M) \times (\fa^*)^2$
such that 
$\Diff_{G}\left(I(\chi_\triv, \lambda_1), I(\sigma, \lambda_2)\right)\neq \{0\}$.
By the duality theorem, it is equivalent to
classifying $(\sigma, \lambda_1, \lambda_2) 
\in \Irr(M) \times (\fa^*)^2$ such that
$\Hom_{\fg, B}
\left(M(\sigma^\ssv, -\lambda_2), 
M(\chi_\triv, -\lambda_1)\right) \neq \{0\}$,
as
\begin{equation}\label{eqn:duality3}
\Diff_{G}\left(I(\chi_\triv, \lambda_1), I(\sigma, \lambda_2)\right)
\simeq
\Hom_{\fg, B}
\left(M(\sigma^\ssv, -\lambda_2), 
M(\chi_\triv, -\lambda_1)\right).
\end{equation}

Lemma \ref{lem:char} below shows that it suffices to consider the case 
$\sigma = \chi \in \Irr(M)_{\mathrm{char}}$.
For simplicity we set $\Irr(M)' = \Irr(M)\setminus \Irr(M)_{\mathrm{char}}$.

\begin{lem}\label{lem:char}
For $(\chi, \lambda, \nu) 
\in \Irr(M)_{\mathrm{char}} \times (\fa^*)^2$,
we have
\begin{equation*}
\Hom_{\fg, B}\left(M(\sigma,\lambda), M(\chi,\nu) \right)
= \{0\} \quad \text{for any $\sigma \in \Irr(M)'$}.
\end{equation*}
\end{lem}

\begin{proof}
Assume that there exists a quadruple
$(\sigma_0, \chi_0, \lambda_0, \nu_0) \in \Irr(M)'\times
\Irr(M)_{\mathrm{char}} \times (\fa^*)^2$ 
such that $\Hom_{\fg, B}\left(M(\sigma_0, \lambda_0),M(\chi_0, \nu_0)\right) \neq \{0\}$.
As $\sigma_0 \in \Irr(M)'$, we write 
$\sigma_0 = \text{span}_\C\{v_1,\ldots, v_n\}$ for some $n > 1$.
Then since $\fb = \fa \oplus \fn$ acts on $\sigma_0$ trivially, we have
$M(\sigma_0, \lambda_0)\vert_{\fg} = \bigoplus_{j=1}^nM(\lambda_0)_j$,
where $M(\lambda_0)_j \simeq M(\lambda_0)$ is the $\Cal{U}(\fg)$-module induced 
from $\C v_j \otimes \mathbb{1}_{\lambda_0-\rho}$.
Since $M(\chi_0,\nu_0)\vert_\fg \simeq M(\nu_0)$, we have
\begin{equation}\label{eqn:Verma32}
\{0\} 
\neq \Hom_{\fg}\left(M(\sigma_0, \lambda_0)\vert_\fg, M(\chi_0, \nu_0)\vert_\fg \right)
=\bigoplus_{j=1}^n\Hom_{\fg}\left(M(\lambda_0)_j, M(\nu_0) \right).
\end{equation}

Let $\varphi$ be a non-zero homomorphism in 
$\Hom_{\fg, B}\left(M(\sigma_0, \lambda_0),M(\chi_0, \nu_0)\right)$.
Since $M(\nu_0)$ has at most one copy of $M(\lambda_0)$ 
(see, for instance,  \cite[Thm.~4.2]{Hum08}), 
identity \eqref{eqn:Verma32} implies that there exists 
$\ell \in \{1,\ldots, n\}$ and $u_0 \in \Cal{U}(\fg)$ such that 
$\varphi(v_\ell \otimes \mathbb{1}_{\lambda_0-\rho}) = u_0 \otimes \mathbb{1}_{\nu_0-\rho}$
and $\varphi(v_j \otimes \mathbb{1}_{\lambda_0-\rho}) = 0$ for $j \neq \ell$. Thus
we have
$\varphi(\sigma_0 \otimes \mathbb{1}_{\lambda_0-\rho}) = \C u_0 \otimes 
\mathbb{1}_{\nu_0-\rho}$.
This contradicts the assumption that $\varphi$ is an $MA$-homomorphism, that is,
$\dim_\C \varphi(\sigma_0 \otimes \mathbb{1}_{\lambda_0-\rho} )
= \dim_\C\sigma_0 =n> 1$. Now the lemma follows.
\end{proof}

It follows from \eqref{eqn:duality3} 
with $\sigma = \chi \in \Irr(M)_{\mathrm{char}}$
that
\begin{align*}
\Diff_{G}\left(I(\chi_\triv, \lambda_1), I(\chi, \lambda_2)\right)
&\simeq
\Hom_{\fg, B}
\left(M(\chi^{-1}, -\lambda_2), 
M(\chi_\triv, -\lambda_1)\right)\\
&\subset 
\Hom_{\fg}
\left(M(\chi^{-1}, -\lambda_2)\vert_{\fg}, 
M(\chi_\triv, -\lambda_1)\vert_{\fg}\right)\\
&=
\Hom_{\fg}
\left(M( -\lambda_2), 
M(-\lambda_1)\right).
\end{align*}
Then we next briefly recall a well-known fact on 
the classification of parameters $(\lambda, \nu) \in \fa^*$ 
for which $\Hom_{\fg}\left(M(\lambda), M(\nu)\right) \neq \{0\}$.
For the details, see, for instance, \cite{BGG71, BGG75}, 
\cite[Chap.~7]{Dix96}, \cite[Chap.~5]{Hum08}, and \cite{Verma68}.

For the rest of this subsection and Section \ref{subsec:Verma2},
we assume that 
$\fg$ is a complex simple Lie algebra and we fix a Cartan subalgebra $\fh$ of $\fg$.
We also fix an inner product $\IP{\cdot}{\cdot}$ on $\fh^*$.
Let $\Delta$ denote the set of roots of $\fg$ with respect to $\fh$.
Choose a positive system $\gD^+$ for $\gD$ and write $\Pi$ for the 
set of simple roots for $\gD^+$. We write $\fb = \fh \oplus \fn$ for
the Borel subalgebra corresponding to $\gD^+$.
For $\ga \in \gD$ and $\lambda \in \fh^*$, 
we write $s_\ga(\lambda) = \lambda - \IP{\lambda}{\ga^\ssv}\ga$
with $\ga^\ssv =\frac{2}{\IP{\ga}{\ga}}\ga$,
the coroot of $\ga$. 

We first recall from the literature the notion of a \emph{link} between 
weights.

\begin{defn}
[Bernstein--Gelfand--Gelfand]
\label{def:Link}
Let $\gl, \lambda \in \fh^*$ and $\gb_1, \ldots, \gb_t \in \gD^+$. Set $\lambda_0 = \lambda$
and $\lambda_i = s_{\gb_i} \cdots s_{\gb_1}(\lambda)$ for $1 \leq i \leq t$.
We say that the sequence $(\gb_1, \ldots, \gb_t)$ \emph{links} $\lambda$ to $\gl$ if 
the following two conditions are satisfied:
\begin{enumerate}
\item[(1)] $\lambda_t = \gl$;
\item[(2)] $\IP{\lambda_{i-1}}{\gb_i^{\ssv}} \in \Z_{\geq 0}$ for $1\leq i \leq t$.
\end{enumerate}
\end{defn}

It is well known that the Verma module $M(\gl)$ has a unique irreducible quotient,
to be denoted by $L(\gl)$,
and also that $\dim_\C\Hom_{\fg}(M(\nu),M(\lambda)) \leq 1$.
The following celebrated result of BGG--Verma shows when 
$\dim_\C\Hom_{\fg}(M(\nu),M(\lambda)) = 1$.

\begin{thm}
[BGG--Verma]
\label{thm:BGGV}
The following three conditions on $\gl, \lambda \in \fh^*$ are equivalent:
\begin{enumerate}
\item[(i)] $\dim_\C\Hom_{\fg}(M(\nu),M(\lambda)) =1;$
\item[(ii)] $L(\gl)$ is a composition factor of $M(\lambda);$
\item[(iii)] there exists a sequence $(\gb_1, \ldots, \gb_t)$ with $\gb_i \in \gD^+$
that links $\lambda$ to $\gl$.
\end{enumerate}
\end{thm}

To determine $\chi \in \Irr(M)_{\mathrm{char}}$ such that 
$\Hom_{\fg, B}
\left(M(\chi^{-1}, -\nu), 
M(\chi_\triv, -\lambda)\right) \neq \{0\}$ in the setting of Section \ref{subsec:Verma0},
observe that 
the homomorphism $\varphi \in \Hom_\fg(M(-\nu),M(-\lambda))$ is 
determined by the image $\varphi(1\otimes \mathbb{1}_{-\nu-\rho})$
of highest weight vector $1\otimes \mathbb{1}_{-\nu-\rho}$ of 
$M(-\nu)$ in $M(-\lambda)$.
Let $u_0\in\Cal{U}(\bar{\fn})$ such that
$\varphi(1\otimes \mathbb{1}_{-\nu-\rho}) 
=u_0 \otimes \mathbb{1}_{-\lambda-\rho}$.
Since each root space is a one-dimensional representation of $M$,
the group $M$ acts on $\C u_0$ as a character $\chi_0$.
This is the character we look for.
In the next subsection we then discuss how to find such $u_0$
as a construction of a homomorphism between Verma modules.

\subsection{Construction of homomorphisms 
between Verma modules}
\label{subsec:Verma2}

As in \eqref{eqn:n-inv}, we set
\begin{equation*}
M(\lambda)^\fn:=\{u \otimes \mathbb{1}_{\lambda -\rho}\in 
M(\lambda) : X \cdot (u\otimes \mathbb{1}_{\lambda-\rho}) = 0 
\text{ for all $X \in \fn$}\}.
\end{equation*}
As $\fn$ is generated by the root vectors $X_\ga$ for $\ga \in \Pi$, it follows that
\begin{equation*}
M(\lambda)^\fn=\{u \otimes \mathbb{1}_{\lambda -\rho}\in 
M(\lambda) : X_\ga \cdot (u\otimes \mathbb{1}_{\lambda-\rho}) = 0 
\text{ for all $\ga \in \Pi$}\}.
\end{equation*}
The elements $u\otimes \mathbb{1}_{\lambda-\rho} \in M(\lambda)^\fn$ are called
\emph{singular vectors} of $M(\lambda)$.
As
$\Hom_{\fg}(M(\nu), M(\lambda)) \simeq \Hom_{\fh}(\C_{\nu -\rho}, M(\lambda)^\fn)$,
to construct a homomorphism $\varphi \in \Hom_{\fg}(M(\nu), M(\lambda))$,
it suffices to find a singular vector 
$u\otimes \mathbb{1}_{\lambda-\rho} \in M(\lambda)^\fn$ 
with weight $\nu-\rho$, namely, the element 
$u\otimes \mathbb{1}_{\lambda-\rho} \in M(\lambda)$ 
satisfying the following two conditions:
\begin{enumerate}
\item[(C1)] the element $u$ has weight $\nu-\lambda$;
\item[(C2)] $X_{\ga} \cdot (u\otimes \mathbb{1}_{\lambda-\rho})=0$ for all $\ga \in \Pi$.
\end{enumerate}

When $\nu$ is given by $\nu = s_{\ga}(\lambda)$ 
with $\IP{\lambda}{\cga} \in 1+\Z_{\geq 0}$ for some $\ga \in \Pi$,
the singular vector of $M(\lambda)$ with weight $\nu-\rho$ is easy to find.
For the proof of the next proposition see, for instance, \cite[Prop.\ 1.4]{Hum08}.

\begin{prop}\label{prop:map}
Given $\lambda \in \fh^*$ and $\ga \in \Pi$, 
suppose that $k:= \IP{\lambda}{\cga} \in 1+ \Z_{\geq 0}$.
Then $X_{-\ga}^k \otimes \mathbb{1}_{\lambda-\rho}$ 
is a singular vector of $M(\lambda)$ with weight $-k\ga + (\lambda-\rho)$.
Consequently, up to scalar multiple, the map $\varphi
\in \Hom_{\fg}(M(s_\ga(\lambda)), M(\lambda))$ is given by
\begin{equation*}
1\otimes \mathbb{1}_{s_\ga(\lambda)-\rho}
\mapsto X^k_{-\ga}\otimes \mathbb{1}_{\lambda-\rho}.
\end{equation*}
\end{prop}

\subsection{Recipe of classification and construction of
$(\fg, B)$-homomorphisms between Verma modules}
\label{subsec:recipe2}

Here, for the sake of convenience, 
we summarize the classification and construction of the
$(\fg,B)$-homomorphism from $M(\chi^{-1},-\nu)$ to $M(\chi_\triv, -\lambda)$
for fixed $\lambda \in \fa^*$.
Via the duality theorem,
these are equivalent to those of 
the intertwining differential operator $\D_{u}\in 
\Diff_G\left(I(\chi_\triv, \lambda), I(\chi, \nu)\right)$.
In this subsection we use the notation in Section \ref{subsec:Verma0}.
\vskip 0.1in

\noindent
\textsf{Recipe for $(\fg, B)$-homomorphisms between Verma modules}.
Fix $\lambda \in \fa^*$.
\vskip 0.1in

\noindent
\textsf{Step H1:}
Classify $\nu \in \fa^*$ with $\nu \neq \lambda$ such that 
\begin{equation*}
\Hom_{\fg}\left(M(-\nu), M(-\lambda)\right)\neq \{0\}.
\end{equation*}
We remark that, by the BGG--Verma theorem (Theorem \ref{thm:BGGV}),
we have 
\begin{equation*}
\#\{\nu\in \fa^*: \Hom_{\fg}(M(-\nu),M(-\lambda)) \neq \{0\} \}<\infty,
\end{equation*}
where $\#S$ denotes the cardinality of a given set $S$.
\vskip 0.2in

\noindent
\textsf{Step H2:}
For each $ \nu \in \fa^*$ classified in Step H1,
construct a homomorphism
\begin{equation*}
\varphi_{(-\nu,-\lambda)} \in \Hom_{\fg}\left(M(-\nu), M(-\lambda)\right),
\end{equation*}
that is, determine $\bar{u}_{\nu} \in \Cal{U}(\bar{\fn})$ such that 
\begin{equation*}
\varphi_{(-\nu,-\lambda)}
(1\otimes \mathbb{1}_{-\nu-\rho})\in \C \bar{u}_{\nu}\otimes \mathbb{1}_{-\lambda-\rho}.
\end{equation*}
\vskip 0.15in

\noindent
\textsf{Step H3:}
For each $ \nu \in \fa^*$ classified in Step H1,
observe the adjoint action $\Ad$ of 
$M$ on $\C \bar{u}_\nu$ to  
determine the character 
$\chi_{-\nu} \in \Irr(M)_{\mathrm{char}}$, such that
\begin{equation}\label{eqn:hom}
\varphi_{(-\nu,-\lambda)} \in 
\Hom_{\fg, B}(M(\chi^{-1}_{-\nu}, -\nu), M(\chi_\triv, -\lambda)).
\end{equation}

\section{Application to $\widetilde{SL}(3,\mathbb{R})$}
\label{sec:SL3}

In this section, toward the later applications, we discuss the structure of 
$\widetilde{SL}(3,\R)$, the non-linear double cover of $SL(3,\R)$.
The characters of $\wM$ and the polynomial realization of the irreducible representations
of  $\wK$ are also discussed.

\subsection{Notation and normalizations}
\label{subsec:prelim}

We start with the notation and normalizations for $\widetilde{SL}(3,\R)$.
Let $\wG= \widetilde{SL}(3,\R)$ with Lie algebra $\fg_0 = \f{sl}(3,\R)$.
Take the Cartan involution $\gt\colon \fg_0 \to \fg_0$ to be
$\gt(U)=-U^t$. We then write $\fg_0= \fk_0 \oplus \fs_0$ for the Cartan decomposition 
of $\fg_0$ with respect to $\gt$, where
$\fk_0$ and $\fs_0$ are as usual the $+1$ and $-1$ eigenspaces of $\gt$, respectively.
We have $\fk_0 = \f{so}(3)\simeq \f{su}(2)$.

Let $\fa_0$ be the maximal abelian subspace of $\fs_0$ defined by
$\fa_0:=\text{span}_\R\{E_{ii}-E_{i+1, i+1}: i=1, 2\}$,
where $E_{ij}$ are the matrix units.
We also define a nilpotent subalgebra $\fn_0$ of $\fg_0$
by $\fn_0:=\text{span}_\R\{E_{12}, E_{23}, E_{13}\}$.
Then $\fb_0:=\fa_0 \oplus \fn_0$ is a minimal parabolic subalgebra of $\fg_0$.

Let $\wK$, $A$, and $N$ be the analytic subgroups of $\wG$ with Lie algebras 
$\fk_0$, $\fa_0$, and $\fn_0$, respectively, 
so that  $\wG=\wK AN$ is an Iwasawa decomposition of $\wG$.
We write $\wM := Z_{\wK}(\fa_0)$.
Then $\wB := \wM AN$ is a minimal parabolic subgroup of $\wG$ with Lie algebra $\fb_0$. 

For real Lie algebra $\mathfrak{y}_0$, we express its complexification by 
$\mathfrak{y}$.
The complexification $\mathfrak{b}= \mathfrak{a} \oplus \mathfrak{n}$ of 
the minimal parabolic subalgebra $\fb_0=\fa_0\oplus \fn_0$ 
is a Borel subalgebra of $\fg= \f{sl}(3,\C)$.
We write $\Delta\equiv \Delta(\mathfrak{g},\mathfrak{a})$ 
for the set of roots of $\mathfrak{g}$ with respect to 
$\mathfrak{a}$ and denote by $\Delta^+$ the positive system corresponding to $\fb$.
Let $\Pi = \{\ga,\gb\}$ be the set of simple roots
for the positive system $\Delta^+$.
We fix $\ga$ and $\gb$ in such a way that 
the root spaces $\fg_\ga$ and $\fg_\gb$ are given by
$\fg_{\ga} = \C E_{12}$ and $\fg_{\gb}=\C E_{23}$.
We write $\rho$ for half the sum of the positive roots.

Define $X, Y \in \fg$ by
\begin{equation}\label{eqn:XY}
X = \begin{pmatrix}
0 & 0 & 0\\
1 & 0 & 0\\
0 & 0 & 0
\end{pmatrix}
\quad \text{and} \quad
Y = \begin{pmatrix}
0 & 0 & 0\\
0 & 0 & 0\\
0 & 1 & 0
\end{pmatrix}.
\end{equation}
Then $X$ and $Y$ are root vectors of $-\ga$ and $-\beta$, respectively.
The nilpotent radical $\bar{\fn}$ opposite to $\fn$ is then given as the spanned
space of $\{X, Y, [X,Y]\}$.

\subsection{Characters $\wchi_{(\eps, \eps')}$ of $\wM$}
\label{subsec:M}

As described in Section \ref{subsec:recipe2},
the characters $\wchi$ of $\wM \subset \widetilde{SL}(3,\R)$ play a key 
role to construct intertwining differential operators $\Cal{D}$.
In this subsection we describe the characters $\wchi$ 
via the characters of the linear group $M \subset SL(3,\R)$.
To the end we first aim to identify the elements of $\wM$ with those of $M$
in a canonical way.

We start with the identifications of $\fk_0= \f{so}(3)$ and $\fk=\f{so}(3,\C)$
with $\f{su}(2)$ and $\f{sl}(2,\C)$, respectively.
First observe that $\f{su}(2)$ is spanned by the three matrices
\begin{equation*}
A_1:=\begin{pmatrix}
\sqrt{-1} & 0\\
0 & -\sqrt{-1}
\end{pmatrix},\quad
A_2:=
\begin{pmatrix}
0 & 1\\
-1 & 0
\end{pmatrix},\quad
A_3:=
\begin{pmatrix}
0 & \sqrt{-1}\\
\sqrt{-1} & 0
\end{pmatrix}
\end{equation*}
with commutation relations
$[A_1, A_2] = 2A_3$,
$[A_1, A_3] = -2 A_2$, and
$[A_2, A_3] = 2A_1$.
On the other hand,
the Lie algebra $\fk_0=\f{so}(3)$ is spanned by
\begin{equation*}
B_1:=
\begin{pmatrix}
0 & 0 & -1\\
0 & 0 & 0\\
1 & 0 &0
\end{pmatrix}, \quad
B_2:=
\begin{pmatrix}
0 & 0 & 0\\
0 & 0 & -1\\
0 & 1 &0
\end{pmatrix}, \quad
B_3:=
\begin{pmatrix}
0 & -1 & 0\\
1 & 0 & 0\\
0 & 0 &0
\end{pmatrix}
\end{equation*}
with commutation relations
$[B_1,B_2]=B_3$,
$[B_1,B_3]=-B_2$, and
$[B_2, B_3] = B_1$.
Thus $\fk_0$ can be identified with $\f{su}(2)$ via the map
\begin{equation*}
\Omega\colon \fk_0 \stackrel{\sim}{\To} \f{su}(2),
\quad B_j \longmapsto \tfrac{1}{2}A_j \quad \text{for $j=1, 2, 3$}.
\end{equation*}

Let $Z_+$, $Z_-$, $Z_0$ be the elements of $\fk$ defined by
\begin{equation}\label{eqn:ZH}
Z_+:= B_2 - \sqrt{-1}B_3,\quad
Z_-:=-(B_2+\sqrt{-1}B_3), \quad
Z_0:=[Z_+, Z_-].
\end{equation}
We set
\begin{equation}\label{eqn:sl2E}
E_+:=\begin{pmatrix}0 & 1\\ 0 & 0\end{pmatrix}, \quad
E_-:=\begin{pmatrix}0 & 0 \\ 1 & 0 \end{pmatrix}, \quad
E_0:=\begin{pmatrix}1 & 0 \\ 0 & -1 \end{pmatrix}.
\end{equation}
Since $A_2-\sqrt{-1}A_3 =2E_+$ and
$-(A_2 + \sqrt{-1}A_3)= 2E_-$,
one may identify $\fk=\f{so}(3,\C)$ 
with $\f{sl}(2,\C)$ via the map
\begin{equation}\label{eqn:sl2}
\Omega_\C\colon \fk \stackrel{\sim}{\To} \f{sl}(2,\C),
\quad 
Z_j \longmapsto E_j \quad \text{for $j=+, - , 0$.}
\end{equation}

The subgroup $\wM = Z_{\wK}(\fa_0)$ is isomorphic to 
the quaternion group $Q_8$, a non-commutative group of order $8$. 
Since $\wK$ is isomorphic to $SU(2)$, 
we realize $\wM$ as a subgroup of $SU(2)$ by
\begin{equation*}
\wM\simeq \{\pm \wm_0, \;  \pm \wm_1, \; \pm \wm_2, \; \pm \wm_3\},
\end{equation*}
where
\begin{equation}\label{eqn:wM}
\wm_0=
\begin{pmatrix}
1&0\\
0&1
\end{pmatrix}, \;\;
\wm_1=
\begin{pmatrix}
\sqrt{-1}&0\\
0&-\sqrt{-1}
\end{pmatrix}, \;\;
\wm_2=
\begin{pmatrix}
0&1\\
-1&0
\end{pmatrix}, \;\;
\wm_3=
\begin{pmatrix}
0&\sqrt{-1}\\
\sqrt{-1}&0
\end{pmatrix}.
\end{equation}
\vskip 0.05in
\noindent
Let $M$ be a subgroup of $SO(3) \subset SL(3,\R)$ defined by
\begin{equation*}
M = \{ m_0, \; m_1, \; m_2, \; m_3\},
\end{equation*}
where
\begin{equation*}
m_0=
\begin{pmatrix}
1& 0 & 0\\
0 & 1 & 0\\
0 & 0 & 1
\end{pmatrix},\;
m_1=
\begin{pmatrix}
-1& 0 & 0\\
0 & 1 & 0\\
0 & 0 & -1
\end{pmatrix},\;
m_2=
\begin{pmatrix}
1& 0 & 0\\
0 & -1 & 0\\
0 & 0 & -1
\end{pmatrix},\;
m_3=
\begin{pmatrix}
-1& 0 & 0\\
0 & -1 & 0\\
0 & 0 & 1
\end{pmatrix}.
\end{equation*}

The adjoint action $\Ad$ of $SU(2)$ on $\f{su}(2)$ yields a two-to-one covering map
$SU(2)\to SO(3)$. 
We realize $\Ad(SU(2))$ as a matrix group with respect to the ordered basis
$\{A_2, A_1, A_3\}$ of $\f{su}(2)$ in such a way that $\wm_j$ are mapped to 
$m_j$ for $j = 0, 1, 2, 3$.
Lemma \ref{lem:Omega} below shows that the map $\pm m_j \mapsto m_j$
respects the Lie algebra isomorphism 
$\Omega_\C \colon \fk \stackrel{\sim}{\to} \f{sl}(2,\C)$.

\begin{lem}\label{lem:Omega}
For $Z \in \fk$, we have
\begin{equation*}
\Omega_\C(\Ad(m_j)Z) = \Ad(\wm_j) \Omega_\C(Z) 
\quad
\textnormal{for\; $j=0, 1, 2, 3$}.
\end{equation*}
\end{lem}
\begin{proof}
By the Lie algebra isomorphism 
$\Omega_\C \colon \fk \stackrel{\sim}{\to} \f{sl}(2,\C)$
in \eqref{eqn:sl2}, in order to prove the lemma, it suffices to show that
$\Omega_\C(\Ad(m_j)Z_k) = \Ad(\wm_j)E_k$
for $j=0, 1, 2, 3$ and $k = +, - , 0$.
One can easily check that these identities indeed hold.
\end{proof}

For $\eps, \eps' \in \{\pm\}$,
we define a character $\chi_{(\eps,\eps')}\colon M \to \{\pm 1\}$ of $M$  by
\begin{equation*}
\chi_{(\eps, \eps')}
(\mathrm{diag}(a_1, a_2, a_3))
:=|a_1|_{\eps}\;|a_3|_{\eps'},
\end{equation*}
where
$|a|_+:=|a|$ and $|a|_-:=a$.
Via the character $\chi_{(\eps, \eps')}$ of $M$,
we define a character $\wchi_{(\eps, \eps')}\colon \wM \to \{\pm 1\}$ of
$\wM$ by
\begin{equation*}
\wchi_{(\eps,\eps')}(\pm \wm_j) := \chi_{(\eps, \eps')}(m_j) 
\quad
\textnormal{for\; $j=0, 1, 2, 3$}.
\end{equation*}
\vskip 0.1in \noindent
We often abbreviate $\wchi_{(\eps,\eps')}$ as $(\eps,\eps')$. 
Table \ref{table:char} illustrates the character table for $(\eps,\eps') =\wchi_{(\eps,\eps')}$.
With the characters $(\eps, \eps')$, the set $\Irr(\wM)$ of 
equivalence classes of irreducible representations of $\wM$ is given as follows:
\begin{equation}\label{eqn:char}
\Irr(\wM)= \{ \pp, \pmi, \mip, \mm, \mathbb{H}\},
\end{equation}
where $\mathbb{H}$ is the unique genuine $2$-dimensional representation 
of $\wM \simeq Q_8$.

\begin{table}[h]
\caption{Character table for $(\eps,\eps')$}
\begin{center}
\renewcommand{\arraystretch}{1.2} 
{
\begin{tabular}{|c|c|c|c|c|}
\hline
& $\pm \wm_0$ & $\pm\wm_1$ & $\pm\wm_2$ & $\pm\wm_3$\\
\hline
$\pp$ & $1$ & $1$ & $1$ & $1$ \\
\hline
$\pmi$ & $1$  & $-1$  & $-1$ & $1$\\
\hline
$\mip$ & $1$ & $-1$ & $1$ & $-1$\\
\hline
$\mm$ & $1$ & $1$ & $-1$ & $-1$\\
\hline
\end{tabular}
}
\end{center}
\label{table:char}
\end{table}%

\subsection{Polynomial realization of the irreducible representations of $\wK$}
\label{subsec:K}

As indicated in the recipe in Section \ref{subsec:recipe1}, to determine the 
$\wK$-type formula of $\Cal{S}ol_{(u; \lambda)}(\sigma)_{\wK}$,
a realization of irreducible representations $\delta$ of $\wK$ 
is chosen. 
In the present situation that  $\wK \simeq SU(2)$,
we realize $\Irr(\wK)$ as 
$\Irr(\wK) \simeq \{(\pi_n, \Pol_n[t]) : n \in \Z_{\geq 0}\}$
with $\Pol_n[t] := \{p(t) \in \Pol[t] : \deg p(t) \leq n\}$,
where $\Pol[t]$ is the space of polynomials of one variable $t$ with complex coefficients.
The representation $\pi_n$ of $SU(2)$ on $\Pol_n[t]$ is defined by
\begin{equation}\label{eqn:pin}
\left(\pi_n(g)p\right)(t) := (ct+d)^n p\left(\frac{at+b}{ct+d}\right)
\quad \text{for} \quad
g = \begin{pmatrix}
a & b\\ c& d
\end{pmatrix}^{-1}.
\end{equation}
The elements $\wm_j$ for $j=1,2, 3$ of 
$\wM\subset SU(2)$ defined in \eqref{eqn:wM}
act on $\Pol_n[t]$ as follows:
\begin{equation}\label{eqn:pi-wM}
\wm_1\colon p(t) \mapsto  (\sqrt{-1})^n p(-t); \quad
\wm_2\colon p(t) \mapsto t^np\left(-\frac{1}{t}\right); \quad
\wm_3\colon p(t) \mapsto (-\sqrt{-1}t)^n p\left(-\frac{1}{t}\right).
\end{equation}

Let $d\pi_n$ be the differential of the representation $\pi_n$. 
Then \eqref{eqn:sl2E} and \eqref{eqn:pin} imply that we have
\begin{equation}\label{eqn:Epm}
\dpin(E_+) = -\frac{d}{dt} \quad \text{and} \quad 
\dpin(E_-) = -nt + t^2\frac{d}{dt}. 
\end{equation}
As usual we extend $d\pi_n$ complex-linearly to $\f{sl}(2,\C)$ and also naturally to
the universal enveloping algebra $\Cal{U}(\f{sl}(2,\C))$.
We then let $\Cal{U}(\fk)$ act on $\Pol_n[t]$ 
via the isomorphism $\Omega_\C \colon \fk \stackrel{\sim}{\to} \f{sl}(2,\C)$ 
in \eqref{eqn:sl2}. For simplicity we write 
$\dpin(F)=\dpin(\Omega_\C(F))$ for $F\in \Cal{U}(\fk)$.
\vskip 0.1in

By Corollary \ref{cor:Frob3} and \eqref{eqn:Ksol2}
with the realization $\Irr(\wK) \simeq \{(\pi_n, \Pol_n[t]) : n \in \Z_{\geq 0}\}$, we have 
\begin{equation}\label{eqn:Soln}
\Cal{S}ol_{(u; \lambda)}(\sigma)_{\wK}
\simeq 
\bigoplus_{n =0}^\infty
\Pol_n[t] \otimes 
\mathrm{Hom}_{\wM}\left(
\Sol_{(u)}(n), \sigma\right),
\end{equation}
where
\begin{equation*}
\Sol_{(u)}(n):=\{ p(t) \in \Pol_n[t] : d\pi_n(\tau(u^{\flat}))p(t) = 0\}.
\end{equation*}
Here $\tau\colon \fg \to \fg$ is the conjugation with respect to 
the real form $\fg_0$.
Since $\fg_0$ is the split real form of $\fg$, we have 
$\tau(u^\flat)=u^\flat$. Thus,
\begin{equation}\label{eqn:Soln2}
\Sol_{(u)}(n)=\{ p(t) \in \Pol_n[t] : d\pi_n(u^{\flat})p(t) = 0\}.
\end{equation}

By using \eqref{eqn:Soln} and \eqref{eqn:Soln2}, 
we shall determine the $\wK$-type formulas for 
intertwining differential operators for the cases of
$\lambda = -\rho$ and $\lambda=-(1/2)\rho$ in Sections \ref{sec:rho} and \ref{sec:wrho},
respectively.

\section{The case of infinitesimal character $\rho$}
\label{sec:rho}

In this section, in accordance with the recipes given in Sections 
\ref{subsec:recipe1} and \ref{subsec:recipe2},
we determine the $\wK$-type formula of the solution space 
$\Cal{S}ol_{(u;\lambda)}(\sigma)$ with $\lambda= -\rho$ 
for $\wG = \widetilde{SL}(3,\R)$.
This is done in Theorem \ref{thm:XandY}.
We continue with the notation and normalizations from the previous section.

\subsection{Classification and construction of intertwining differential operators}
\label{subsec:IDO-rho1}

Our first goal is to classify and construct intertwining differential operators,
namely, to achieve Steps H1--H3 in the recipe in Section \ref{subsec:recipe2}.
As the first step we start by classifying $\nu \in  \fa^*$ with $\nu \neq \rho$
such that 
$\Hom_{\fg}(M(-\nu), M(\rho))\neq \{0\}$.

\begin{lem}\label{lem:rho-map1}
The following are equivalent on $\nu \in \fa^*$ with $\nu\neq \rho$.
\begin{enumerate}
\item[(i)] $\Hom_\fg(M(-\nu), M(\rho))\neq \{0\}$.
\item[(ii)] $\nu = \pm \ga, \pm \gb, \rho$.
\end{enumerate}
\end{lem}

\begin{proof}
The BGG--Verma theorem (Theorem \ref{thm:BGGV}) shows that
the following are all homomorphisms (that are not proportional to the identities)
obtained from $M(\rho)$,
where $\varphi_{(\mu_1,\mu_2)}$ denotes the homomorphism 
from $M(\mu_1)$ to $M(\mu_2)$.
\begin{equation}\label{eqn:homs}
\begin{gathered}
\xymatrix@R=2em{
&M(-\gb) 
\quad\ar[r]^{ \varphi_{(-\gb,\gb)}}  
\ar[ddr]_{\qquad \varphi_{(-\ga,\gb)}\hspace{0.2cm}}
& \quad M(\gb)  \ar[dr]^{ \varphi_{(\gb,\rho)}}&\\
M(-\rho) 
\ar[ur]^{ \varphi_{(-\rho,-\gb)}}
\ar[dr]_{\varphi_{(-\rho,-\ga)}}
\ar[rrr]^{\qquad \varphi_{(-\rho,\rho)}\hspace{4.5cm}}
& & & M(\rho)\\
&M(-\ga) \quad \ar[r]_{\varphi_{(-\ga,\ga)}}  
\ar[uur]_{\hspace{0.2cm} \varphi_{(-\gb,\ga)}}
&\quad M(\ga)  \ar[ur]_{ \varphi_{(\ga,\rho)}}&\\
}
\end{gathered}
\end{equation}

\end{proof}

The next step is to construct the homomorphism 
$\varphi_{(-\nu,\rho)} \in \Hom_{\fg}(M(-\nu), M(\rho))$ for 
$\nu = \pm \ga, \pm \gb, \rho$.
Let $X$ and $Y$ be the root vectors of $-\ga$ and $-\gb$
defined in \eqref{eqn:XY}, respectively.

\begin{lem}\label{lem:maps2}
Up to scalar multiple,
the image of $1\otimes \mathbb{1}_{-\nu-\rho} \in M(-\nu)$
under the map $\varphi_{(-\nu,\, \rho)} \in \Hom_{\fg}(M(-\nu), M(\rho))$ 
for $\nu = \pm \ga, \pm \gb, \rho$
is given as follows.
\begin{alignat*}{4}
&\varphi_{(\gb,\rho)}&&: &&\; X\otimes \mathbb{1}_0,
&&\;\varphi_{(-\gb,\rho)}:  Y^2X\otimes \mathbb{1}_0,\\
& \varphi_{(\ga,\rho)}&&: &&\;  Y\otimes \mathbb{1}_0,
&&\;\varphi_{(-\ga,\rho)}: X^2Y\otimes \mathbb{1}_0,\\
&\varphi_{(-\rho,\rho)}&&: &&\;  XY^2X\otimes \mathbb{1}_0 
(=YX^2Y \otimes \mathbb{1}_0)
\end{alignat*}
\end{lem}

\begin{proof}
Proposition \ref{prop:map} gives all maps $\varphi_{(\mu_1,\mu_2)}$
in \eqref{eqn:homs}.
By composing these maps,  we obtain the lemma.
\end{proof}

As the third step, for each $\nu=\pm \ga, \pm \gb, \rho$, 
we next determine a character 
$\wchi_{-\nu}\in \Irr(\wM)_{\text{char}}$
such that $\Hom_{\fg, \wB}(M(\wchi_{-\nu}^{-1}, -\nu), M(\pp, \rho)) \neq \{0\}$.
As the third step, for each $\nu=\pm \ga, \pm \gb, \rho$, 
we next determine a character 
$\wchi_{-\nu}\in \Irr(\wM)_{\text{char}}$
such that $\Hom_{\fg, \wB}(M(\wchi_{-\nu}^{-1}, -\nu), M(\pp, \rho)) \neq \{0\}$.
Let $\bar{u}_{\nu}$ be the element of $\Cal{U}(\bar\fn)$ determined in Lemma \ref{lem:maps2}
such that $\varphi_{(-\nu,\rho)}(1\otimes \mathbb{1}_{-\nu-\rho}) 
= \bar{u}_{\nu} \otimes \mathbb{1}_0$. We have
\begin{alignat}{4}\label{eqn:u-nu}
&\nu=-\beta&&: &&\; \bar{u}_{\beta} = X, \; &&\;\nu=\beta: \bar{u}_{-\beta}=Y^2X, \nonumber\\
&\nu=-\alpha&&: &&\; \bar{u}_{\alpha} = Y, \; &&\; \nu=\alpha: \bar{u}_{-\alpha}=X^2Y,\\
&\nu=\rho&&: &&\; \bar{u}_{-\rho} = XY^2X (=YX^2Y).\nonumber
\end{alignat}

\begin{lem}\label{lem:char-rho}
Let $\nu \in \{\pm \ga, \pm \gb, \rho\}$.
Then the character $\wchi_{-\nu} \in \Irr(\wM)_{\mathrm{char}}$
for which we have $\Hom_{\fg, \wB}(M(\wchi_{-\nu}^{-1}, -\nu), M(\pp, \rho)) \neq \{0\}$ 
is given as follows.
\begin{enumerate}
\item[(a)]  $\nu=\pm\gb:$ $\wchi_{-\nu} = \pmi$.
\item[(b)]  $\nu=\pm\ga:$ $\wchi_{-\nu} =\mip$.
\item[(c)]  $\nu=\rho \hspace{0.35cm} :$ $\wchi_{-\nu} =\pp$.
\end{enumerate} 
\end{lem}

\begin{proof}
We wish to check that, via the adjoint action, the subgroup $\wM$ acts on 
$\C \bar{u}_{\nu}$ by the proposed character.
Since the adjoint action of $\wM$ factors through $M$,
it suffices to consider the adjoint action of $M$ on $\C \bar{u}_{\nu}$.
Now a direct computation yields the lemma.
\end{proof}

For $\varphi_{(-\nu,\rho)}(1\otimes \mathbb{1}_{-\nu-\rho})
=\bar{u}_{\nu}\otimes \mathbb{1}_0$,
write $\D_{\bar{u}_{\nu}} = R(\bar{u}_{\nu})$. 
We then obtain the following.

\begin{lem}\label{lem:diff-rho}
For $\nu=\pm \ga, \pm \gb, \rho$, 
let $\wchi_{-\nu}$ be the character of $\wM$ determined in Lemma \ref{lem:char-rho}.
Then we have
\begin{equation*}
\Diff_{\wG}\left(I(\pp, -\rho), I(\wchi_{-\nu}, \nu)\right) =\C \D_{\bar{u}_{\nu}}.
\end{equation*}
\end{lem}

\begin{proof}
This is an immediate consequence of Lemmas \ref{lem:maps2} and \ref{lem:char-rho}
and the duality theorem.
\end{proof}

It follows from \eqref{eqn:u-nu} that we have
\begin{alignat*}{2}
&\Cal{S}ol_{(X;-\rho)}(\sigma) &&\subset \Cal{S}ol_{(\bar{u}_\nu;-\rho)}(\sigma) 
\quad \text{for $\nu = -\gb, -\rho$,}\\
&\Cal{S}ol_{(Y;-\rho)}(\sigma) &&\subset  \Cal{S}ol_{(\bar{u}_\nu;-\rho)}(\sigma) 
\quad \text{for $\nu = -\ga, -\rho$.}
\end{alignat*}
Then, in the next subsection,
we consider the $\wK$-type formulas of the solution spaces
$\Cal{S}ol_{(X;-\rho)}(\sigma)_{\wK}$, $\Cal{S}ol_{(Y;-\rho)}(\sigma)_{\wK}$,
 and $\Cal{S}ol_{(X, Y;-\rho)}(\sigma)_{\wK}
 :=\Cal{S}ol_{(X;-\rho)}(\sigma)_{\wK} \cap \Cal{S}ol_{(Y;-\rho)}(\sigma)_{\wK}$.
 
\subsection{$\wK$-type formulas of solution spaces }
\label{subsec:IDO-rho2}

We now aim to find the $\wK$-type formulas of 
$\Cal{S}ol_{(X;-\rho)}(\sigma)_{\wK}$, $\Cal{S}ol_{(Y;-\rho)}(\sigma)_{\wK}$,
 and $\Cal{S}ol_{(X, Y;-\rho)}(\sigma)_{\wK}$.
In order to determine them,
as in the recipe in Section \ref{subsec:recipe1},
we start by finding $\tau(u^\flat)=u^\flat\in \fk_0$ for $u = X, Y \in \bar{\fn}_0$.
Let $Z_+$ and $Z_-$ be the nilpotent elements 
in the $\f{sl}(2)$-triple of $\fk = \f{so}(3,\C)$ defined in \eqref{eqn:ZH}.

\begin{lem}\label{lem:Kpic-rho}
We have $X^\flat = \frac{\sqrt{-1}}{2}(Z_+ + Z_-)$ 
and $Y^\flat = \frac{1}{2}(Z_+ - Z_-)$.
\end{lem}

\begin{proof}
As $X$ and $Y$ are root vectors for $-\ga$ and $-\gb$ with $\ga, \gb \in \Pi$, 
respectively, we have $X^\flat = X+\gt(X)$ and $Y^\flat = Y+\gt(Y)$,
where $\gt$ is the Cartan involution defined by $\gt(U) = -U^t$.
A direct computation then concludes the lemma.
\end{proof}

The next step is to choose a realization of $\delta \in \Irr(\wK)$.
As described in Section \ref{subsec:K},
we realize $\Irr(\wK)$ as $\Irr(\wK) \simeq \{(\pi_n, \Pol_n[t]) : n \in \Z_{\geq 0}\}$.
The explicit formulas for the operators $\dpin(X^\flat)$ and $\dpin(Y^\flat)$
are given as follows.

\begin{lem}\label{lem:dpin-rho}
We have 
\begin{equation*}
\dpin(X^\flat)= -\frac{\sqrt{-1}}{2} ((1-t^2)\frac{d}{dt} + nt) \;\;
 \text{and} \;\;
\dpin(Y^\flat)= -\frac{1}{2} ((1+t^2)\frac{d}{dt} - nt).
\end{equation*}
\end{lem}

\begin{proof}
It follows from \eqref{eqn:sl2} and \eqref{eqn:Epm} that
\begin{equation}\label{eqn:Zpm}
\dpin(Z_+) = -\frac{d}{dt} \quad \text{and} \quad 
\dpin(Z_-) = -nt + t^2\frac{d}{dt}. 
\end{equation}
Now the proposed identities follow from Lemma \ref{lem:Kpic-rho}.
\end{proof}

Recall from \eqref{eqn:Soln} and \eqref{eqn:Soln2} that
the space $\Cal{S}ol_{(u; -\rho)}(\sigma)_{\wK}$ 
of $\wK$-finite solutions to $\Cal{D}_u\otimes \mathrm{id}_\sigma$ 
for $u=X,Y$ is decomposed as
\begin{equation*}
\Cal{S}ol_{(u; -\rho)}(\sigma)_{\wK}
\simeq 
\bigoplus_{n =0}^\infty
\Pol_n[t] \otimes 
\mathrm{Hom}_{M}\left(
\Sol_{(u)}(n), \sigma\right),
\end{equation*}
where
\begin{equation*}
\Sol_{(u)}(n)=\{ p(t) \in \Pol_n[t] : d\pi_n(u^{\flat})p(t) = 0\}.
\end{equation*}
We next wish to determine for what $n \in \Z_{\geq 0}$
the solution space $\Sol_{(u)}(n)$ is non-zero.

\begin{prop}\label{prop:SolXandY}
For $u= X, Y$, we have
$\Sol_{(u)}(n) \neq \{0\}$
if and only if  $n \in 2\Z_{\geq0}$.
Moreover, for $n \in 2\Z_{\geq 0}$, we have 
\begin{equation*}
\Sol_{(X)}(n) = \C(1-t^2)^{\frac{n}{2}} \quad \text{and} \quad
\Sol_{(Y)}(n) = \C(1+t^2)^{\frac{n}{2}}.
\end{equation*}
\end{prop}

\begin{proof}
We only discuss about $\Sol_{(X)}(n)$;
the assertion for $\Sol_{(Y)}(n)$ can be drawn from a similar argument.
By Lemma \ref{lem:dpin-rho}, it suffices to solve the differential equation
$((1-t^2)\frac{d}{dt} + nt)p(t) = 0$.
By separation of variables, one can easily check that the solution of the differential
equation has to be proportional to $(1-t^2)^{\frac{n}{2}}$. 
The assertion then follows from a simple observation that 
$(1-t^2)^{\frac{n}{2}} \in \Pol_n[t]$ if and only if $n$ is even.
\end{proof}

As in \eqref{eqn:SysSol}, we define
\begin{equation*}
\Sol_{(X,Y)}(n):= \Sol_{(X)}(n) \cap \Sol_{(Y)}(n).
\end{equation*}

\begin{cor}\label{cor:SolXY}
We have
$\Sol_{(X, Y)}(n) \neq \{0\}$
if and only if  $n = 0$.
Moreover we have 
$\Sol_{(X, Y)}(0) =\C \cdot 1$.
\end{cor}

\begin{proof}
This is an immediate consequence of Proposition \ref{prop:SolXandY}.
\end{proof}

We next show the $\wM$-representation
on $\Sol_{(X)}(n)$, $\Sol_{(Y)}(n)$, and $\Sol_{(X,Y)}(n)$.

\begin{prop}\label{prop:Mrep-rho}
As an $\wM$-representation, we have the following.
\begin{enumerate}
\item[(1)] $n\equiv 0 \pmod 4:$ 
$\Sol_{(X)}(n) \hspace{0.26cm} \simeq \pp$, 
$\Sol_{(Y)}(n) \simeq \pp$.
\item[(2)] $n\equiv 2 \pmod 4:$
$\Sol_{(X)}(n) \hspace{0.26cm} \simeq \pmi$,
$\Sol_{(Y)}(n) \simeq \mip$.
\item[(3)] $n=0:$\hspace{1.53cm}
$\Sol_{(X,Y)}(0)  \simeq \pp$.
\end{enumerate}
\end{prop}

\begin{proof}
For $u=X, Y, (X,Y)$, let $p_{(u)}(t)$ be the polynomial 
determined in Proposition \ref{prop:SolXandY} and Corollary \ref{cor:SolXY}
such that $\Sol_{(u)}(n) = \C p_{(u)}(t)$. 
The assertions then follow from a direct observation of the transformation laws
\eqref{eqn:pi-wM} of $\wm_j \in \wM$ 
on $p_{(u)}(t)$ with Table \ref{table:char}.
\end{proof}

As the last step of the recipe,
we determine $n \in \Z_{\geq 0}$ 
such that $\Hom_{\wM}(\Sol_{(u)}(n), \sigma)\neq \{0\}$
for given $\sigma \in \Irr(\wM)$.

\begin{prop}\label{prop:Krep-rho}
For $u = X, Y, (X,Y)$, the following are equivalent
on $\sigma \in \Irr(\wM)$:
\begin{enumerate}
\item[(i)] $\Hom_{\wM}(\Sol_{(u)}(n), \sigma)\neq \{0\};$
\item[(ii)] $\dim_\C \Hom_{\wM}(\Sol_{(u)}(n), \sigma)=1$.
\end{enumerate}
Further, for each $u = X, Y, (X,Y)$, we have the following.
\begin{enumerate}
\item If $\sigma \neq \pp, \pmi$, then $\Hom_{\wM}(\Sol_{(X)}(n), \sigma) = \{0\}$
for all $n \in \Z_{\geq0}$. Moreover, 
\begin{enumerate}
\item $\Hom_{\wM}(\Sol_{(X)}(n), \pp)\neq \{0\}$ $\iff$ $n \equiv 0 \pmod 4;$
\item $\Hom_{\wM}(\Sol_{(X)}(n), \pmi)
\neq \{0\}$ $\iff$ $n \equiv 2 \pmod 4$.
\end{enumerate}
\vskip 0.1in

\item If $\sigma \neq \pp, \mip$, then $\Hom_{\wM}(\Sol_{(Y)}(n), \sigma) = \{0\}$
for all $n \in \Z_{\geq0}$. Moreover,
\begin{enumerate}
\item $\Hom_{\wM}(\Sol_{(Y)}(n), \pp)\neq \{0\}$ $\iff$ $n \equiv 0 \pmod 4;$
\item $\Hom_{\wM}(\Sol_{(Y)}(n), \mip)
\neq \{0\}$ $\iff$ $n \equiv 2 \pmod 4$.
\end{enumerate}
\vskip 0.1in

\item If $\sigma \neq \pp$, then $\Hom_{\wM}(\Sol_{(X,Y)}(n), \sigma) = \{0\}$
for all $n \in \Z_{\geq0}$. Moreover,
\begin{enumerate}
\item[] $\Hom_{\wM}(\Sol_{(X,Y)}(n), \pp)\neq \{0\} \iff n=0.$
\end{enumerate}
\end{enumerate}
\end{prop}

\begin{proof}
The assertions easily follow from
Proposition \ref{prop:SolXandY}, Corollary \ref{cor:SolXY}, and 
Proposition \ref{prop:Mrep-rho}.
\end{proof}

Here is a summary of the results that we have obtained in this section.

\begin{thm}\label{thm:XandY}
For $\sigma \in \Irr(\wM)$, the following hold.
\begin{enumerate}
\item
$\Cal{S}ol_{(X;-\rho)}(\sigma)_{\wK} \hspace{0.3cm}
\neq \{0\}  \iff \sigma = \pp, \pmi$.
\item
$\Cal{S}ol_{(Y;-\rho)}(\sigma)_{\wK} 
 \hspace{0.33cm}
\neq \{0\} \iff \sigma = \pp, \mip$.
\item
$\Cal{S}ol_{(X,Y;-\rho)}(\sigma)_{\wK} \neq \{0\} \iff \sigma = \pp$.
\end{enumerate}
Moreover, for $\sigma \in \Irr(\wM)$ 
such that $\Cal{S}ol _{(u;-\rho)}(\sigma)\neq \{0\}$,
the $\wK$-type formula of 
$\Cal{S}ol _{(u;-\rho)}(\sigma)_{\wK}$ is determined as follows.
\begin{enumerate}
\item[(a)] $u=X:$\hspace{0.75cm}
$\displaystyle{\Cal{S}ol _{(X;-\rho)}(\pp)_{\wK} 
\hspace{0.27cm}
\simeq \bigoplus_{\ell=0}^\infty \Pol_{4\ell}}[t]$ 
\; and \;
$\displaystyle{\Cal{S}ol _{(X;-\rho)}(\pmi)_{\wK} 
\simeq \bigoplus_{\ell=0}^\infty \Pol_{4\ell+2}}[t]$.
\vskip 0.05in

\item[(b)] $u=Y:$\hspace{0.8cm}
$\displaystyle{\Cal{S}ol _{(Y;-\rho)}(\pp)_{\wK} 
\hspace{0.3cm}
\simeq \bigoplus_{\ell=0}^\infty \Pol_{4\ell}}[t]$ 
\; and \;
$\displaystyle{\Cal{S}ol _{(Y;-\rho)}(\mip)_{\wK} 
\simeq \bigoplus_{\ell=0}^\infty \Pol_{4\ell+2}}[t]$.
\vskip 0.05in

\item[(c)] $u=(X,Y):$
$\displaystyle{\Cal{S}ol _{(X, Y;-\rho)}(\pp)_{\wK} \simeq \Pol_{0}[t]}$.
\end{enumerate}
\end{thm}

\begin{proof}
We only demonstrate the proof for 
$\Cal{S}ol _{(X;-\rho)}(\pp)_{\wK}$; the other cases can be shown similarly.
By \eqref{eqn:Soln}, we have
\begin{equation*}
\Cal{S}ol_{(X; -\rho)}(\sigma)_{\wK}
\simeq 
\bigoplus_{n =0}^\infty
\Pol_n[t] \otimes 
\mathrm{Hom}_{\wM}\left(
\Sol_{(X)}(n), \sigma\right).
\end{equation*}
It then follows from Proposition \ref{prop:Krep-rho} that
$\Cal{S}ol_{(X; -\rho)}(\sigma)_{\wK} = \{0\}$ 
for $\sigma \neq \pp, \pmi$
and, for $\sigma = \pp, \pmi$, we have
\begin{equation*}
\Cal{S}ol_{(X; -\rho)}(\pp)_{\wK}
\simeq 
\bigoplus_{n \equiv 0 \hspace{-0.1in}  \pmod 4 }
\Pol_n[t] 
\quad
\text{and}
\quad
\Cal{S}ol_{(X; -\rho)}(\pmi)_{\wK}
\simeq 
\bigoplus_{n \equiv 2 \hspace{-0.1in}  \pmod 4 }
\Pol_n[t].
\end{equation*}
Now the assertions follow.
\end{proof}

We now give a proof of Theorem \ref{thm:Ktype-rho}.

\begin{proof}[Proof of Theorem \ref{thm:Ktype-rho}]
Since
$\Cal{S}ol_{(u; -\rho)}(\sigma)_{\wK}$ is dense in
$\Cal{S}ol_{(u; -\rho)}(\sigma)$, we have
$\Cal{S}ol_{(u; -\rho)}(\sigma)\neq \{0\}$
if and only if $\Cal{S}ol_{(u; -\rho)}(\sigma)_{\wK}\neq \{0\}$.
As $(\pi_n, \Pol_n[t]) \simeq V_{(\frac{n}{2})}$,
the assertions follow from Theorem \ref{thm:XandY}.
\end{proof}

\section{The case of infinitesimal character $\wrho$}
\label{sec:wrho}

The aim of this section is to determine the $\wK$-type formula
of the solution space $\Cal{S}ol_{(u;\lambda)}(\sigma)$ 
for $\lambda=-(1/2)\rho$. 
This is achieved in  Theorem \ref{thm:XcY1}.
For the rest of this section we write $\wrho=(1/2)\rho$.

\subsection{Classification and construction of intertwining differential operators}
\label{subsec:IDO-wrho1}

As for the case of $\rho$,
we start by classifying $\nu \in  \fa^*$ with $\nu \neq \wrho$ such that 
$\Hom_{\fg}(M(-\nu), M(\wrho))\neq \{0\}$.

\begin{lem}
The following are equivalent on $\nu \in \fa^*$ with $\nu\neq \wrho$.
\begin{enumerate}
\item[(i)] $\Hom_\fg(M(-\nu), M(\wrho))\neq \{0\}$.
\item[(ii)] $\nu =  \wrho$.
\end{enumerate}
\end{lem}

\begin{proof}
A simple observation of the BGG--Verma theorem (Theorem \ref{thm:BGGV}) 
for $\lambda = \wrho$.
\end{proof}

The next step is to construct a non-zero homomorphism 
$\varphi_{(-\wrho, \wrho)} \in \Hom_{\fg}(M(-\wrho), M(\wrho))$.
To do so, we first prepare some notation.
Let $\Cal{U}_r(\bar{\fn})$
be the subspace of $\Cal{U}(\bar{\fn})$ 
that is spanned by at most $r$ elements of $\bar{\fn}$.
Also, let $S^r(\bar{\fn})$ denote the subspace of the symmetric algebra 
$S(\bar{\fn})$
of $\bar{\fn}$ spanned by $r$ elements of $\bar{\fn}$. 
Then we have
\begin{equation}\label{eqn:sym}
\Cal{U}_r(\bar{\fn}) =\sym^r(S^r(\bar{\fn})) \oplus \Cal{U}_{r-1}(\bar{\fn}),
\end{equation}
where $\sym^r \colon S^r(\bar{\fn}) \to \Cal{U}_r(\bar{\fn})$ is the symmetrization map.
Recall from \eqref{eqn:XY} that $X$ and $Y$ are root vectors of $-\ga$ and $-\gb$,
respectively.
With the notation we show the following.

\begin{lem}\label{lem:map-wrho}
Up to scalar multiple,
the map $\varphi_{(-\wrho, \wrho)} \in \Hom_{\fg}(M(-\wrho), M(\wrho))$
is given by
\begin{equation*}
1\otimes \mathbb{1}_{-(\wrho+\rho)}
\longmapsto (XY+YX)\otimes \mathbb{1}_{-\wrho}.
\end{equation*}
\end{lem}

\begin{proof}
It is first remarked that, in contrast to Lemma \ref{lem:maps2},
we cannot apply Proposition \ref{prop:map} in the present case,
as $-\wrho \neq s_{\gamma}(\wrho)$ for $\gamma = \ga, \gb$.
We thus consider the conditions (C1) and (C2) in Section \ref{subsec:Verma2}.
Let $u_0 \otimes \mathbb{1}_{-\wrho}$ be the image of 
$1\otimes \mathbb{1}_{-(\wrho + \rho)}$ under the map
$\varphi_{(-\wrho, \wrho)}\in \Hom_{\fg}(M(-\wrho), M(\wrho))$.
Suppose that $u_0$ satisfies the condition (C1), namely,
$u_0$ has weight $-\rho=-\ga-\gb$. Then $u_0$ is a linear combination
of $XY$ and $YX$; in particular, 
$u_0 \in \Cal{U}_2(\bar{\fn})$.
By \eqref{eqn:sym}, this implies that 
$u_0 \in \C(XY+YX)$. 
A direct computation shows that 
$(XY+YX)\otimes  \mathbb{1}_{-\wrho}$ 
satisfies the condition (C2), that is, 
it is annihilated by $X_{\ga}$ and $X_{\gb}$ in $M(\wrho)$.
(Here one may take $X_{\ga}$ and $X_{\gb}$ to be $X_{\ga}=E_{12}$ and $X_{\gb}=E_{23}$,
as $\fg_\ga=\C E_{12}$ and $\fg_\gb = \C E_{23}$; see Section \ref{subsec:prelim}.)
Now the lemma follows.
\end{proof}

We next determine $\wchi\in \Irr(\wM)_{\text{char}}$
such that $\Hom_{\fg, \wB}(M(\wchi^{-1}, -\wrho), M(\pp, \wrho)) \neq \{0\}$.

\begin{lem}\label{lem:char-wrho}
The following conditions on $\wchi \in \Irr(\wM)_{\mathrm{char}}$ are equivalent.
\begin{enumerate}
\item[(i)]  $\Hom_{\fg, \wB}(M(\wchi^{-1}, -\wrho), M(\pp, \wrho)) \neq \{0\}$.
\item[(ii)] $\wchi=\mm$.
\end{enumerate} 
\end{lem}

\begin{proof}
Since the proof is similar to the one for Lemma \ref{lem:char-rho},
we omit the proof.
\end{proof}

For simplicity we write 
\begin{equation}\label{eqn:XcY}
\XcY =XY+YX,
\end{equation}
so that $\D_{\XcY} = R(X)R(Y)+R(Y)R(X)$.

\begin{lem} 
We have
$\Diff_{\wG}\left(I(\pp, -\wrho), I(\mm, \wrho)\right) =\C \D_{\XcY}$.
\end{lem}

\begin{proof}
This is an immediate consequence of 
Lemmas \ref{lem:map-wrho} and \ref{lem:char-wrho} and the 
duality theorem.
\end{proof}

\subsection{$\wK$-type formulas of the solution spaces
$\Cal{S}ol_{(\XcY;-\wrho)}(\sigma)_{\wK}$} 
\label{subsec:IDO-wrho2}

We now aim to obtain the $\wK$-type formula of 
$\Cal{S}ol_{(\XcY;-\wrho)}(\sigma)_{\wK}$.
As for the $\rho$ case, we start by finding $(\XcY)^\flat$.

\begin{lem}\label{lem:Kpic-wrho}
We have $(\XcY)^\flat =X^\flat Y^\flat + Y^\flat X^\flat
 = \frac{\sqrt{-1}}{2}(Z_+^2 - Z_-^2)$.
\end{lem}

\begin{proof}
A direct computation shows that 
$\iota((X^\flat Y^\flat + Y^\flat X^\flat)\otimes \mathbb{1}) 
= (\XcY)\otimes \mathbb{1}_{-\wrho}$,
where $\iota$ is the isomorphism in \eqref{eqn:GK}.
This concludes the first identity.
The second identity easily follows from Lemma \ref{lem:Kpic-rho}.
\end{proof}

For the next we find the explicit formula of $\dpin((\XcY)^\flat)$.

\begin{lem}\label{lem:dpin-wrho}
We have
\begin{equation}\label{eqn:dpi7}
\dpin((\XcY)^\flat)
= \frac{\sqrt{-1}}{2} ((1-t^4)\frac{d^2}{dt^2}+2(n-1)t^3\frac{d}{dt} - n(n-1)t^2).
\end{equation}
\end{lem}

\begin{proof}
By Lemma \ref{lem:Kpic-wrho}, we have
\begin{equation*}
\dpin((\XcY)^\flat) = \frac{\sqrt{-1}}{2}(\dpin(Z_+)^2 -\dpin(Z_-)^2).
\end{equation*}
The desired identity then follows from 
\eqref{eqn:Zpm} with a direct computation.
\end{proof}

We next want to determine for what $n \in \Z_{\geq 0}$
the solution space $\Sol_{(\XcY)}(n)$ is non-zero.
Let ${}_2F_1[a,b,c;x]$ denote the Gauss hypergeometric series.
We put
\begin{equation}\label{eqn:uv}
u_n(t):= {}_2F_1[-\frac{n}{4}, -\frac{n-1}{4}, \frac{3}{4};t^4]
\;\; \text{and} \;\;
v_n(t):= t\, {}_2F_1[-\frac{n-1}{4}, -\frac{n-2}{4}, \frac{5}{4};t^4].
\end{equation}
It is remarked that $u_n(t)$ and $v_n(t)$ form a fundamental set of solutions
to Euler's hypergeometric equation $D[a, b,c;x]f(x)=0$ with 
$a=-\frac{n}{4}$, $b=-\frac{n-1}{4}$, and $c=\frac{3}{4}$.
(For some details see Appendix.)
In particular $u_n(t)$ and $v_n(t)$ are linearly independent.

\begin{prop}\label{prop:SolXcY}
We have
$\Sol_{(\XcY)}(n) \neq \{0\}$
if and only if  $n \equiv 0, 1, 2 \pmod 4$.
Moreover the solution space 
$\Sol_{(\XcY)}(n)$ is given as follows.
\begin{enumerate}
\item $n \equiv 0 \pmod 4:$ $\Sol_{(\XcY)}(n) = \C u_n(t)$.
\item $n \equiv 1 \pmod 4:$ $\Sol_{(\XcY)}(n) = \C u_n(t) \oplus \C v_n(t)$.
\item $n \equiv 2 \pmod 4:$ $\Sol_{(\XcY)}(n) = \C v_n(t)$.
\end{enumerate}
\end{prop}

Since the proof involves some classical facts on the Gauss hypergeometric 
series ${}_2F_1[a,b,c;x]$, we give the proof in 
Section \ref{subsec:appendix2} in Appendix.
Then, by taking the assertions of Proposition \ref{prop:SolXcY} as granted,
we next show the $\wM$-representations
on $\Sol_{(\XcY)}(n)$ for $n \equiv 0, 1, 2 \pmod 4$.

\begin{prop}\label{prop:Mrep-wrho}
As an $\wM$-representation, we have the following.
\begin{enumerate}
\item[(1)] $n\equiv 0 \pmod 4:$ 
$\Sol_{(\XcY)}(n) \simeq \pp$. 
\item[(2)] $n\equiv 1 \pmod 4:$
$\Sol_{(\XcY)}(n)  \simeq \mathbb{H}$.
\item[(3)] $n\equiv 2 \pmod 4:$
$\Sol_{(\XcY)}(n)  \simeq \mm$.
\end{enumerate}
\end{prop}

\begin{proof}
We discuss (1) and (2) only; 
the assertion (3) can be shown similarly to (1).
We start with the assertion (1).
Suppose that $n\equiv 0 \pmod 4$. Write $n=4k$ for some $k \in \Z_{\geq 0}$ so that
\begin{equation}\label{eqn:uk}
u_{4k}(t)= {}_2F_1[-k, -k+\frac{1}{4}, \frac{3}{4};t^4].
\end{equation}
By Proposition \ref{prop:SolXcY}, we have
$\Sol_{(\XcY)}(4k) = \C u_{4k}(t)$; in particular, $\wM$ acts on $\Sol_{(\XcY)}(4k)$ 
as a character $(\eps,\eps')$.
As $\wm_3 = \wm_1\wm_2$, to determine the character $(\eps,\eps')$,
it suffices to consider the actions of $\wm_1$ and $\wm_2$ on $\C u_{4k}(t)$.  
We claim that both $\wm_1$ and $\wm_2$ act on $\C u_{4k}(t)$ trivially.
First it is easy to see that the action of $\wm_1$ on $\C u_{4k}(t)$ is trivial.
Indeed, it follows from \eqref{eqn:pi-wM} and \eqref{eqn:uk} that we have
\begin{equation*}
\wm_1\colon u_{4k}(t)\longmapsto (\sqrt{-1})^{4k}u_{4k}(-t)=u_{4k}(t).
\end{equation*}
In order to show that the action of $\wm_2$ is also trivial, observe that
$\wm_2$ transforms $u_{4k}(t)$ as
\begin{equation*}
\wm_2\colon u_{4k}(t)\longmapsto t^{4k}\, u_{4k}(-\frac{1}{t})
=t^{4k}u_{4k}(\frac{1}{t}).
\end{equation*}
Since $\Sol_{(\XcY)}(4k)=\C u_{4k}(t)$ is an $\wM$-representation,
there exists a constant $c \in \C$ such that 
$t^{4k}u_{4k}(\frac{1}{t})=cu_{4k}(t)$.
In particular we have $u_{4k}(1) = c u_{4k}(1)$. 
It follows from a general fact (Fact \ref{fact:Gauss} in Section \ref{subsec:appendix1})
on the Gauss hypergeometric series ${}_2F_1[a,b,c;x]$ that 
\begin{equation*}
u_{4k}(1) = {}_2F_1[-k, -k+\frac{1}{4}, \frac{3}{4};1] = 
\frac{\Gamma(\frac{3}{4})\Gamma(2k+\frac{1}{2})}
{\Gamma(k+\frac{3}{4})\Gamma(k+\frac{1}{2})}\neq 0.
\end{equation*}
Thus we have $c=1$. Therefore $\wm_2$ also acts on $\C u_{4k}(t)$ trivially.

In order to show the assertion (2), suppose that $n\equiv 1 \pmod 4$.
Since there is only a unique irreducible $2$-dimensional $\wM$-representation 
(see \eqref{eqn:char}), it suffices to show that the $\wM$-representation 
$\C u_n(t) \oplus \C v_n(t)$ is irreducible.
We show by contradiction.
Assume the contrary, that is, there exists a non-trivial proper $\wM$-invariant subspace 
$\{0\}\neq V \varsubsetneq \C u_n(t) \oplus \C v_n(t)$. 
Then $V$ is of the form
$V=\C(ru_n(t) + sv_n(t))$ for some $(r,s) \in \C^2$ with $(r,s) \neq (0,0)$.
Observe that, as $n\equiv 1 \pmod 4$, we have $(\sqrt{-1})^n = \sqrt{-1}$.
Thus $\wm_1$ transforms $u_n(t)$ and $v_n(t)$ as
\begin{equation*}
u_n(t) \longmapsto (\sqrt{-1})^n u_n(-t) = \sqrt{-1}u_n(t), \quad \text{and} \quad
v_n(t)  \longmapsto (\sqrt{-1})^n v_n(-t) = -\sqrt{-1}v_n(t).
\end{equation*}
Therefore $ru_n(t) + sv_n(t)$ is transformed by $\wm_1$ as
\begin{equation}\label{eqn:rusv1}
ru_n(t) + sv_n(t) \longmapsto \sqrt{-1}(ru_n(t) - sv_n(t)).
\end{equation}
Observe that, as $V$ is a 1-dimensional $\wM$-representation,
$\wM$ acts on $V$ as a character $(\eps,\eps')$.
In particular $\wm_1$ acts on $V$ by $\pm 1$, that is,
\begin{equation}\label{eqn:rusv2}
ru_n(t) + sv_n(t) \longmapsto \pm(ru_n(t) + sv_n(t)).
\end{equation}
Equations \eqref{eqn:rusv1} and \eqref{eqn:rusv2} 
imply that 
$(\sqrt{-1}-1)ru_n(t)-(\sqrt{-1}+ 1)sv_n(t) =0$ or
$(\sqrt{-1}+1)ru_n(t)-(\sqrt{-1}-1)sv_n(t) =0$.
Since $u_n(t)$ and $v_n(t)$ are linearly independent, this concludes
that $(r,s) =0$, which contradicts the choice of $(r,s)\neq (0,0)$.
Hence $\C u_n(t) \oplus \C v_n(t)$ is irreducible.
\end{proof}

\begin{prop}\label{prop:Krep-wrho}
The following are equivalent on $\sigma \in \Irr(\wM)$:
\begin{enumerate}
\item[(i)] $\Hom_{\wM}(\Sol_{(\XcY)}(n), \sigma)\neq \{0\};$
\item[(ii)] $\dim_\C \Hom_{\wM}(\Sol_{(\XcY)}(n), \sigma)=1;$
\item[(iii)] $\sigma = \pp, \mathbb{H}$, \mm.
\end{enumerate}
Moreover we have the following.
\begin{enumerate}
\item $\Hom_{\wM}(\Sol_{(\XcY)}(n), \pp)\neq \{0\}$ 
$\iff$ $n \equiv 0 \pmod 4$.
\item $\Hom_{\wM}(\Sol_{(\XcY)}(n), \mathbb{H})
\hspace{0.63cm} 
\neq \{0\}$ 
$\iff$ $n \equiv 1 \pmod 4$.
\item $\Hom_{\wM}(\Sol_{(\XcY)}(n), \mm)
\neq \{0\}$ 
$\iff$ $n \equiv 2 \pmod 4$.
\end{enumerate}
\end{prop}

\begin{proof}
The proposition easily follows from 
Propositions \ref{prop:SolXcY} and \ref{prop:Mrep-wrho} with
\eqref{eqn:Soln}.
\end{proof}

As a summary of the results in this section, we obtain the following.

\begin{thm}\label{thm:XcY1}
For $\sigma \in \Irr(\wM)$, we have
\begin{equation*}
\Cal{S}ol_{(\XcY;-\wrho)}(\sigma)_{\wK} \neq \{0\} 
\iff \sigma = \pp,\, \mathbb{H},\, \mm.
\end{equation*}
Moreover, for $\sigma = \pp,\, \mathbb{H},\, \mm$, the $\wK$-type formula 
of $\Cal{S}ol _{(\XcY;-\wrho)}(\sigma)_{\wK}$ is obtained as follows.
\begin{enumerate}
\item[(a)] $\sigma = \pp:$ 
$\displaystyle{\Cal{S}ol _{(\XcY;-\wrho)}(\pp)_{\wK} 
\simeq \bigoplus_{\ell=0}^\infty \Pol_{4\ell}[t]}$.

\item[(b)] $\sigma= \mathbb{H}:$\hspace{0.6cm}
$\displaystyle{\Cal{S}ol _{(\XcY;-\wrho)}(\mathbb{H})_{\wK} \hspace{0.7cm}
\simeq \bigoplus_{\ell=0}^\infty \Pol_{4\ell+1}[t]}$.
\item[(c)] $\sigma= \mm:$ 
$\displaystyle{\Cal{S}ol _{(\XcY;-\wrho)}(\mm)_{\wK} 
\simeq \bigoplus_{\ell=0}^\infty \Pol_{4\ell+2}[t]}$.
\end{enumerate}
\end{thm}

\begin{proof}
Since the proof is similar to the one for Theorem \ref{thm:XandY},
we omit the proof.
\end{proof}

Now we give a proof of Theorem \ref{thm:XcY}.

\begin{proof}[Proof of Theorem \ref{thm:XcY}]
As for Theorem \ref{thm:Ktype-rho},
the assertions follow from Theorem \ref{thm:XcY1}.
\end{proof}


\section{Appendix}\label{sec:appendix}

The purpose of this short appendix is to give a proof of Proposition \ref{prop:SolXcY}, 
in which we determine the space $\Sol_{(\XcY)}(n)$ 
of polynomial solutions to the differential equation $\dpin((\XcY)^\flat)p(t)=0$. 
To do so we observe Euler's hypergeometric differential equation.

\subsection{Gauss hypergeometric series ${}_2F_1[a,b,c;x]$}
\label{subsec:appendix1}

We start by recalling some well-known facts on the Gauss hypergeometric series
\begin{equation*}
{}_2F_1[a,b,c;x]:=\sum_{k=0}^\infty \frac{(a)_k(b)_k}{(c)_k} \cdot \frac{x^k}{k!},
\end{equation*}
where $(a)_k=\frac{\Gamma(a+k)}{\Gamma(a)}$. 
It is clear that when $c \notin \Z_{\geq 0}$,
the series ${}_2F_1[a,b,c;x]$ is a polynomial if and only if 
either $-a\in \Z_{\geq 0}$ or $-b \in \Z_{\geq 0}$. 
Moreover if $-a \in \Z_{\geq 0}$ and $-b, -c \notin \Z_{\geq 0}$, then
$\deg{}_2F_1[a,b,c;x]=a$. For the proof of the following identity,
see, for instance, \cite[Thm.~2.2.2]{AAR99}.

\begin{fact}[Gauss, 1812]\label{fact:Gauss}
For $\text{Re}(c-a-b)>0$, we have 
\begin{equation*}
{}_2F_1[a,b,c;1]
=\sum_{k=0}^\infty \frac{(a)_k(b)_k}{(c)_k} 
=\frac{\Gamma(c)\Gamma(c-a-b)}{\Gamma(c-a)\Gamma(c-b)}.
\end{equation*}
\end{fact}

Define a second-order differential operator
\begin{equation*}
D{[a,b,c; x]}:=x(1-x)\frac{d^2}{dx^2} + (c-(a+b+1)x)\frac{d}{dx}- ab,
\end{equation*}
so that the equation
$D{[a,b,c;x]}f(x)=0$ is Euler's hypergeometric differential equation.
We put
\begin{equation}\label{eqn:uv1}
u'_{[a,b,c]}(x):={}_2F_1[a,b,c;x]\quad \text{and}\quad
v'_{[a,b,c]}(x):=x^{1-c}{}_2F_1[a-c+1, b-c+1, 2-c;x].
\end{equation}

It is well-known that if $c \notin \Z$, then
$u'_{[a,b,c]}(x)$ and $v'_{[a,b,c]}(x)$ form 
a fundamental set of solutions to 
Euler's hypergeometric differential equation 
$D_{[a,b,c;\; x]}f(x)=0$.
(See, for instance, \cite[Sect.\ 3]{Oshima13}.)

\subsection{Proof of Proposition \ref{prop:SolXcY}}
\label{subsec:appendix2}

We now define another second-order differential operator
\begin{equation*}
T[n;t]= (1-t^4)\frac{d^2}{dt^2}+2(n-1)t^3\frac{d}{dt} - n(n-1)t^2
\end{equation*}
so that $d\pi_n((\XcY)^\flat) = \frac{\sqrt{-1}}{2}T[n;t]$ (see \eqref{eqn:dpi7}).
Thus, in order to solve the differential equation $d\pi_n((\XcY)^\flat)g(t)=0$,
it suffices to solve $T[n;t]g(t)=0$.
Lemma \ref{lem:HGE} below shows the parameters $a_0, b_0$, and $c_0$
such that the equation $T[n;t]g(t)=0$ is identified with 
Euler's hypergeometric differential equation $D{[a_0,b_0, c_0;x]}f(x)=0$.

\begin{lem}\label{lem:HGE}
By change of variables $x=t^4$, the differential operator
$T[n;t]$ is given by 
$T[n;t] =16t^2D{[-\frac{n}{4}, -\frac{n-1}{4}, \frac{3}{4};t^4]}$.
\end{lem}

\begin{proof}
This follows from a routine computation.
\end{proof}

We set
\begin{equation*}
u_n(t) :=u'_{[-\frac{n}{4}, -\frac{n-1}{4}, \frac{3}{4}]}(t^4)\quad \text{and} \quad
v_n(t) :=v'_{[-\frac{n}{4}, -\frac{n-1}{4}, \frac{3}{4}]}(t^4),
\end{equation*}
namely,
\begin{equation}\label{eqn:uv2}
u_n(t)= {}_2F_1[-\frac{n}{4}, -\frac{n-1}{4}, \frac{3}{4};t^4]
\quad \text{and} \quad
v_n(t)= t{}_2F_1[-\frac{n-1}{4}, -\frac{n-2}{4}, \frac{5}{4};t^4].
\end{equation}
We now give a proof of Proposition \ref{prop:SolXcY}.

\begin{proof}[Proof of Proposition \ref{prop:SolXcY}]
Since $\frac{3}{4} \notin \Z$, 
it follows from Lemma \ref{lem:HGE} that 
the functions
$u_n(t)$ and $v_n(t)$ form a fundamental set of solutions 
to the equation $d\pi_n((\XcY)^\flat)g(t) = \frac{\sqrt{-1}}{2}T[n;t]g(t)=0$.
Moreover \eqref{eqn:uv2} shows that
if $n \equiv 3 \pmod 4$, then neither $u_n(t)$ nor $v_n(t)$ is a polynomial;
for $n \equiv 0, 1, 2 \pmod 4$, we have the following.
\begin{alignat*}{4}
&\mathrm{(1)}\; 
&&n\equiv 0 \pmod 4:&&\; u_n(t)&&  \in \Pol_n[t].\hspace{3.2in}\\
& \mathrm{(2)}\;
&&n\equiv 1 \pmod 4:&& \; u_n(t), v_n(t)&& \in \Pol_n[t].\\
& \mathrm{(3)}\; 
&& n\equiv 2 \pmod 4:&& \; v_n(t)&& \in \Pol_n[t].
\end{alignat*}
Since $\Sol_{(\XcY)}(n)$ is the space of polynomial solutions to 
$d\pi_n((\XcY)^\flat)g(t) =0$, this concludes the proposition.
\end{proof}


\textbf{Acknowledgements.}
Part of this research was conducted during a visit of the first author
at the Department of Mathematics of Aarhus University
and a visit of the second author at the Graduate School of 
Mathematical Sciences of the University of Tokyo. 
They are grateful for their support and warm hospitality during their stay.

The authors are thankful to
Toshio Oshima, Toshihiko Matsuki, Kyo Nishiyama, Hiroyuki Ochiai, Kenji Taniguchi, 
and Toshiyuki Kobayashi for fruitful interactions on this work.
Their sincere thanks also go to Anthony Kable for his valuable comments
on a manuscript of this work.

The first author was partially supported by JSPS
Grant-in-Aid for Young Scientists (B) (26800052).


\bibliographystyle{amsplain}
\bibliography{KuOrsted}


\end{document}